\documentclass[12pt,a4paper,english]{amsart}
\usepackage[T1]{fontenc}
\usepackage[latin9]{inputenc}
\usepackage{amsthm}
\usepackage{amsbsy}
\usepackage{amstext}
\usepackage{amssymb}

\makeatletter
\numberwithin{equation}{section} 
\numberwithin{figure}{section} 
\theoremstyle{plain}
\swapnumbers
\newtheorem{thm}{\sc Theorem}[section]
  \theoremstyle{remark}
  \newtheorem*{rem*}{Remark}
  \theoremstyle{plain}
  \newtheorem{cor}[thm]{\sc Corollary}
  \theoremstyle{plain}
  \newtheorem{lem}[thm]{\sc Lemma}
  \theoremstyle{plain}
  \newtheorem{defn}[thm]{\sc Definition}
  \theoremstyle{plain}
  \newtheorem{prop}[thm]{\sc Proposition}

\usepackage[T1]{fontenc}
\usepackage{textcomp}

\usepackage[scaled=0.92]{helvet}
\usepackage{a4wide}

\makeatother

\usepackage{babel}

\begin{document}

\title{Closed orbits of a charge in a weakly exact magnetic field}

\author{Will J. Merry }
\begin{abstract}
We prove that for a weakly exact magnetic system on a closed connected
Riemannian manifold, almost all energy levels contain a closed orbit.
More precisely, we prove the following stronger statements. Let $(M,g)$
denote a closed connected Riemannian manifold and $\sigma\in\Omega^{2}(M)$
a weakly exact $2$-form. Let $\phi_{t}:TM\rightarrow TM$ denote
the magnetic flow determined by $\sigma$, and let $c(g,\sigma)\in\mathbb{R}\cup\{\infty\}$
denote the Ma\~né critical value of the pair $(g,\sigma)$. We prove
that if $k>c(g,\sigma)$, then for every non-trivial free homotopy
class of loops on $M$ there exists a closed orbit of $\phi_{t}$
with energy $k$ whose projection to $M$ belongs to that free homotopy
class. We also prove that for almost all $k<c(g,\sigma)$ there exists
a closed orbit of $\phi_{t}$ with energy $k$ whose projection to
$M$ is contractible. In particular, when $c(g,\sigma)=\infty$ this
implies that almost every energy level has a contractible closed orbit.
As a corollary we deduce that a weakly exact magnetic flow with $[\sigma]\ne0$
on a manifold with amenable fundamental group (which implies $c(g,\sigma)=\infty$)
has contractible closed orbits on almost every energy level. 
\end{abstract}

\address{Department of Pure Mathematics and Mathematical Statistics, University
of Cambridge, Cambridge CB3 0WB, England}

\email{\texttt{w.merry@dpmms.cam.ac.uk}}

\maketitle

\section{Introduction}

Let $(M,g)$ denote a closed connected Riemannian manifold, with tangent
bundle $\pi:TM\rightarrow M$ and universal cover $\tilde{M}$. We
will assume $M$ is \emph{not }simply connected, as otherwise $\tilde{M}=M$
and all results proved in this paper reduce to special cases of the
results in \cite{Contreras2006}. Let $\sigma\in\Omega^{2}(M)$ denote
a closed $2$-form, and let $\tilde{\sigma}\in\Omega^{2}(\tilde{M})$
denote its pullback to the universal cover. In this paper we consider
the case where $\sigma$ is \emph{weakly exact}, that is, when $\tilde{\sigma}$
is exact (this is equivalent to requiring that $\sigma|_{\pi_{2}(M)}=0$),
however we\emph{ }do\emph{ not }assume that $\tilde{\sigma}$ necessarily
admits a \emph{bounded} primitive.\newline 

Let $\omega_{g}$ denote the standard symplectic form on $TM$ obtained
by pulling back the canonical symplectic form $dq\wedge dp$ on $T^{*}M$
via the Riemannian metric. Let $\omega:=\omega_{g}+\pi^{*}\sigma$
denote the \emph{twisted symplectic form }determined by the pair $(g,\sigma)$.
Let $E:TM\rightarrow\mathbb{R}$ denote the energy Hamiltonian $E(q,v)=\frac{1}{2}\left|v\right|^{2}$.
Let $\phi_{t}:TM\rightarrow TM$ denote the flow of the symplectic
gradient of $E$ with respect to $\omega$; $\phi_{t}$ is known as
a \emph{twisted geodesic flow }or a \emph{magnetic flow}. The reason
for the latter terminology is that this flow can be thought of as
modeling the motion of a particle of unit mass and unit charge under
the effect of a magnetic field represented by the $2$-form $\sigma$.
Given $k\in\mathbb{R}^{+}:=\{t\in\mathbb{R}\,:\, t>0\}$, let $\Sigma_{k}:=E^{-1}(k)\subseteq TM$.\newline 

There exists a number $c=c(g,\sigma)\in\mathbb{R}\cup\{\infty\}$
called the \emph{Ma\~né critical value} (see \cite{Mane1996,ContrerasDelgadoIturriaga1997,ContrerasIturriaga1999,BurnsPaternain2002}
or Section \ref{sec:Preliminaries} below for the precise definition)
such that the dynamics of the hypersurface $\Sigma_{k}$ differs dramatically
depending on whether $k<c$, $k=c$ or $k>c$. Moreover $c<\infty$
if and only if $\tilde{\sigma}$ admits a bounded primitive.\newline

In this paper we study the old problem of the existence of closed
orbits on prescribed energy levels. In the case when $\sigma$ is
exact, this problem has been essentially solved by Contreras in \cite{Contreras2006};
see Theorem D in particular, which gives contractible closed orbits
in almost every energy level below the Ma\~né critical value, and
closed orbits in every free homotopy class for any energy level above
the critical value. In the case of surfaces a stronger result is known
to hold: Contreras, Macarini and Paternain have proved in \cite[Theorem 1.1]{ContrerasMacariniPaternain2004}
that in this case \emph{every} energy level admits a closed orbit.
However the case of a \emph{magnetic monopole} (i.e. when $\sigma$
is not exact) remains open, although much progress has been made.
Let us describe now some of these results. A more comprehensive survey
can be found in the introduction to \cite{ContrerasMacariniPaternain2004};
see also \cite{Ginzburg1996} for a introductory account of the problem.

It was proved by Macarini \cite{Macarini2006}, extending an earlier
result of Polterovich \cite{Polterovich1998}, that if $[\sigma]\ne0$
there exist non trivial contractible closed orbits of the magnetic
flow in a sequence of arbitrarily small energy levels. Kerman \cite{Kerman2000}
proved the same result for magnetic fields given by symplectic forms.
This was then recently sharped by Ginzburg and G$\ddot{\mbox{u}}$rel
\cite{GinzburgGurel2009} and finally by Usher \cite{Usher2009},
where it is proved that when $\sigma$ is symplectic contractible
closed orbits exist for all low energy levels. See also Lu \cite{Lu2006}
for another interesting approach to the problem in the case of $\sigma$
symplectic. Perhaps the most general result so far is due to Schlenk
\cite{Schlenk2006} where it is shown that for \emph{any} closed $2$-form
(not necessarily weakly exact), almost every sufficiently small energy
level contains a contractible closed orbit. \newline

The aim of this paper is to extend Theorem D of \cite{Contreras2006}
to the weakly exact case. More precisely, we prove the following.
\begin{thm}
\label{thm:theorem A}Let $(M,g)$ denote a closed connected Riemannian
manifold, and let $\sigma\in\Omega^{2}(M)$ denote a closed $2$-form
whose pullback to the universal cover $\tilde{M}$ is exact. Let $c=c(g,\sigma)\in\mathbb{R}\cup\{\infty\}$
denote the Ma\~né critical value, and let $\phi_{t}$ denote the
magnetic flow defined by $\sigma$. Then:
\begin{enumerate}
\item If $c<\infty$ then for all $k>c$ and for each non-trivial homotopy
class $\nu\in[\mathbb{T},M]$, there exists a closed orbit of $\phi_{t}$
with energy $k$ such that the projection to $M$ of that orbit belongs
to $\nu$.
\item For almost all $k\in(0,c)$, where possibly $c=\infty$, there exists
a contractible closed orbit of $\phi_{t}$ with energy $k$.\emph{ }
\end{enumerate}
\end{thm}
The first statement of Theorem \ref{thm:theorem A} has, under a mild
technical assumption on $\pi_{1}(M)$, previously been proved by Paternain
\cite{Paternain2006}. We use a completely different method of proof
however, which bypasses the need for this additional assumption. For
$c(g,\sigma)<\infty$, the second statement of Theorem \ref{thm:theorem A}
is due to Osuna \cite{Osuna2005}. It should be emphasized that we
believe the main contribution of the present paper is the case $c(g,\sigma)=\infty$
in the second statement. 
\begin{rem*}
We will actually prove a slightly stronger statement than the one
stated above; see Proposition \ref{thm:stronger} below for details.
\end{rem*}
When $\pi_{1}(M)$ is \emph{amenable }and $\sigma$ is not exact,
we always have $c(g,\sigma)=\infty$ (see for instance \cite[Corollary 5.4]{Paternain2006}).
Thus we have the following immediate corollary.
\begin{cor}
Let $(M,g)$ denote a closed connected Riemannian manifold, and let
$\sigma\in\Omega^{2}(M)$ denote a closed non-exact $2$-form whose
pullback to the universal cover $\tilde{M}$ is exact. Suppose $\pi_{1}(M)$
is amenable. Then almost every energy level contains a contractible
closed orbit of the magnetic flow defined by $\sigma$.
\end{cor}
Let us now give a brief outline of our method of attack. Fix a primitive
$\theta$ of $\tilde{\sigma}$, and consider the Lagrangian $L:T\tilde{M}\rightarrow\mathbb{R}$
defined by \[
L(q,v):=\frac{1}{2}\left|v\right|^{2}-\theta_{q}(v).\]
The Euler-Lagrange flow of $L$ is precisely the lifted flow $\tilde{\phi}_{t}:T\tilde{M}\rightarrow T\tilde{M}$
of the magnetic flow $\phi_{t}:TM\rightarrow TM$ (see for example
\cite{ContrerasIturriaga1999}). Recall that the \emph{action} $A(y)$
of the Lagrangian $L$ over an absolutely continuous curve $y:[0,T]\rightarrow\tilde{M}$
is given by \begin{align*}
A(y): & =\int_{0}^{T}L(y(t),\dot{y}(t))dt.\\
 & =\int_{0}^{T}\frac{1}{2}\left|\dot{y}(t)\right|^{2}dt-\int_{y}\theta.\end{align*}
Set \[
A_{k}(y):=\int_{0}^{T}\left\{ L(y(t),\dot{y}(t))+k\right\} dt=A(y)+kT.\]
 A closed orbit of $\tilde{\phi}_{t}$ with energy $k$ can be realized
as a critical point of the functional $y\mapsto A_{k}(y)$. More precisely,
let $\Lambda_{\tilde{M}}$ denote the Hilbert manifold of absolutely
continuous curves $x:\mathbb{T}\rightarrow\tilde{M}$ and consider
$\tilde{S}_{k}:\Lambda_{\tilde{M}}\times\mathbb{R}^{+}\rightarrow\mathbb{R}$
defined by \begin{align*}
\tilde{S}_{k}(x,T): & =\int_{0}^{1}T\cdot L(x(t),\dot{x}(t)/T)dt+kT\\
 & =\int_{0}^{1}\frac{1}{2T}\left|\dot{x}(t)\right|^{2}dt+kT-\int_{x}\theta.\end{align*}
Then the pair $(x,T)$ is a critical point $\tilde{S}_{k}$ if and
only if $y(t):=x(t/T)$ is the projection to $\tilde{M}$ of a closed
orbit of $\tilde{\phi}_{t}$ with energy $k$ (see \cite{ContrerasIturriagaPaternainPaternain2000}).\newline

If $\sigma$ was actually exact then we could define $L$ on $TM$,
instead of just on $T\tilde{M}$. In this case it has been shown in
\cite{ContrerasIturriagaPaternainPaternain2000} that for $k>c(g,\sigma)$,
$\tilde{S}_{k}$ satisfies the Palais-Smale condition and is bounded
below. Standard results from Morse theory \cite[Corollary 23]{ContrerasIturriagaPaternainPaternain2000}
then tell us that $\tilde{S}_{k}$ admits a global minimum, and this
gives us our desired closed orbit. In \cite{Contreras2006} this was
extended to give contractible orbits on almost every energy level
below the critical value. Crucially however these results use compactness
of $M$, and hence are not applicable directly in the weakly exact
case, since then $L$ is defined only on $T\tilde{M}$.\newline 

In the weakly exact case, whilst $\tilde{S}_{k}$ is not well defined
on $TM$, its differential\emph{ }is. This leads to the key observation
of the present work that we can still work directly on $\Lambda_{M}$.
More precisely, we define a functional $S_{k}:\Lambda_{M}\times\mathbb{R}^{+}\rightarrow\mathbb{R}$
with%
\footnote{If $\tilde{\sigma}$ does not admit any bounded primitives $S_{k}$
is only defined on $\Lambda_{0}\times\mathbb{R}^{+}$, where $\Lambda_{0}\subseteq\Lambda_{M}$
is the subset of contractible loops.%
} the property that $(x,T)$ is a critical point of $S_{k}$ if and
only if a lift $\tilde{y}$ to $\tilde{M}$ of the curve $y(t):=x(t/T)$
is the projection to $\tilde{M}$ of a flow line of $\tilde{\phi}_{t}$
with energy $k$. The functional $S_{k}$ is given by \[
S_{k}(x,T):=\int_{0}^{T}\frac{1}{2T}\left|\dot{x}(t)\right|^{2}dt+kT-\int_{C(x)}\sigma,\]
where $C(x)$ is any cylinder with boundary $x(\mathbb{T})\cup x_{\nu}(\mathbb{T})$,
where $x_{\nu}\in\Lambda_{M}$ is some fixed reference loop in the
free homotopy class $\nu\in[\mathbb{T},M]$ that $x$ belongs to.
If $c(g,\sigma)<\infty$, then since $\sigma$ is weakly exact, the
value $\int_{C(x)}\sigma$ is independent of the choice of cylinder
$C(x)$ for any curve $x\in\Lambda_{M}$. In the case $c(g,\sigma)=\infty$,
the value $\int_{C(x)}\sigma$ is independent of the choice of cylinder
only when $x$ is a contractible loop.\newline

The functional $S_{k}$ allows one to extend other results previously
known only for the exact case to the weakly exact case. For instance,
in \cite{Merry2009} we will use $S_{k}$ to establish the short exact
sequence \cite{CieliebakFrauenfelderOancea2009,AbbondandoloSchwarz2009}
between the Rabinowitz Floer homology of a weakly exact twisted cotangent
bundle and the singular (co)homology of the free loop space.\newline 

\begin{flushleft}
\emph{Acknowledgment. }I would like to thank my Ph.D. adviser Gabriel
P. Paternain for many helpful discussions. 
\par\end{flushleft}

\section{\label{sec:Preliminaries}Preliminaries}

\subsection*{The setup}

$\ $\vspace{6 pt}

It will be convenient to view $M$ and $\tilde{M}$ as being embedded
isometrically in some $\mathbb{R}^{d}$ (possible by Nash's theorem).
We will be interested in various spaces of absolutely continuous curves. 

Firstly, given $q_{0},q_{1}\in M$ and $T\geq0$, let $C_{M}^{\textrm{ac}}(q_{0},q_{1};T)$
denote the set of absolutely continuous curves $y:[0,T]\rightarrow M$
with $y(0)=q_{0}$ and $y(T)=q_{1}$. Let \[
C_{M}^{\textrm{ac}}(q_{0},q_{1}):=\bigcup_{T\geq0}C_{M}^{\textrm{ac}}(q_{0},q_{1};T).\]
We can repeat the construction on $\tilde{M}$ to obtain for $q_{0},q_{1}\in\tilde{M}$
sets $C_{\tilde{M}}^{\textrm{ac}}(q_{0},q_{1};T)$ and $C_{\tilde{M}}^{\textrm{ac}}(q_{0},q_{1})$
of curves on $\tilde{M}$. 

Next, consider the space \[
W^{1,2}(\mathbb{R}^{d}):=\left\{ x:I\rightarrow\mathbb{R}^{m}\mbox{ absolutely continuous}\,:\,\int_{0}^{1}\left|\dot{x}(t)\right|^{2}dt<\infty\right\} ,\]
and \[
W^{1,2}(M):=\left\{ x\in W^{1,2}(\mathbb{R}^{d})\,:\, x(I)\subseteq M\right\} .\]
with $W^{1,2}(\tilde{M})$ defined similarly. Here $I:=[0,1]$, a
convention we shall follow throughout the paper.

Let $\Lambda_{\mathbb{R}^{d}}\subseteq W^{1,2}(\mathbb{R}^{d})$ denote
the set of closed loops of class $W^{1,2}$ on $\mathbb{R}^{d}$,
and let $\Lambda_{M}:=W^{1,2}(M)\cap\Lambda_{\mathbb{R}^{d}}$. We
will think of maps $x\in\Lambda_{M}$ as maps $x:\mathbb{T}\rightarrow M$
(here $\mathbb{T}=\mathbb{R}/\mathbb{Z}$, which we shall often identify
with $S^{1}$). Given a free homotopy class $\nu\in[\mathbb{T},M]$,
let $\Lambda_{\nu}\subseteq\Lambda_{M}$ denote the connected component
of $\Lambda_{M}$ consisting of the loops belonging to $\nu$. 

The tangent space to $\Lambda_{\mathbb{R}^{d}}$ at $x\in\Lambda_{\mathbb{R}^{d}}$
is given by\[
T_{x}\Lambda_{\mathbb{R}^{d}}=\left\{ \xi\in W^{1,2}(\mathbb{R}^{d})\,:\,\xi(0)=\xi(1)\right\} .\]
Given $(x,T)\in\Lambda_{M}\times\mathbb{R}^{+}$ we thus have \[
T_{(x,T)}(\Lambda_{M}\times\mathbb{R}^{+})=\left\{ (\xi,\psi)\in W^{1,2}(\mathbb{R}^{d})\times\mathbb{R}\,:\,\xi(0)=\xi(1)\right\} .\]

Let $\left\langle \cdot,\cdot\right\rangle $ denote the standard
Euclidean metric. The metric on $W^{1,2}(\mathbb{R}^{d})$ we will
work with is \[
\left\langle \xi,\zeta\right\rangle _{1,2}:=\left\langle \xi(0),\zeta(0)\right\rangle +\int_{0}^{1}\left\langle \dot{\xi}(t),\dot{\zeta}(t)\right\rangle dt.\]
This defines a metric which we shall denote simply by $\left\langle \cdot,\cdot\right\rangle $
on $W^{1,2}(\mathbb{R}^{d})\times\mathbb{R}^{+}$ by \begin{equation}
\left\langle (\xi,\psi),(\zeta,\chi)\right\rangle :=\left\langle \xi,\zeta\right\rangle _{1,2}+\psi\chi.\label{eq:standard metric}\end{equation}

\subsection*{Ma\~né's critical value}

$\ $\vspace{6 pt}

We now recall the definition of $c(g,\sigma)$, the \emph{critical
value }introduced by Ma\~né in \cite{Mane1996}, which plays a decisive
role in all that follows.\newline

Let us fix a primitive $\theta$ of $\tilde{\sigma}$. Given $k\in\mathbb{R}^{+}$,
we define $A_{k}$ as follows. Let $q_{0},q_{1}\in\tilde{M}$. Define
$A_{k}:C_{\tilde{M}}^{\textrm{ac}}(q_{0},q_{1})\rightarrow\mathbb{R}$
by \[
A_{k}(y):=\int_{0}^{T}\frac{1}{2}\left|\dot{y}(t)\right|^{2}+kT-\int_{y}\theta.\]
We define \emph{Ma\~né's action potential} $m_{k}:\tilde{M}\times\tilde{M}\rightarrow\mathbb{R}\cup\{-\infty\}$
by \[
m_{k}(q_{0},q_{1}):=\inf_{T>0}\inf_{y\in C_{\tilde{M}}^{\textrm{ac}}(q_{0},q_{1};T)}A_{k}(y).\]
Then we have the following result; for a proof see for instance \cite[Proposition 2-1.1]{ContrerasIturriaga1999}
for the first five statements, and \cite[Appendix A]{BurnsPaternain2002}
for a proof of the last statement.
\begin{lem}
Properties of $m_{k}$:
\begin{enumerate}
\item If $k\leq k'$ then $m_{k}(q_{0},q_{1})\leq m_{k'}(q_{0},q_{1})$
for all $q_{0},q_{1}\in\tilde{M}$.
\item For all $k\in\mathbb{R}$ and all $q_{0},q_{1},q_{2}\in\tilde{M}$
we have \[
m_{k}(q_{0},q_{2})\leq m_{k}(q_{0},q_{1})+m_{k}(q_{1},q_{2}).\]

\item Fix $k\in\mathbb{R}$. Then either $m_{k}(q_{0},q_{1})=-\infty$ for
all $q_{0},q_{1}\in\tilde{M}$, or $m_{k}(q_{0},q_{1})\in\mathbb{R}$
for all $q_{0},q_{1}\in\tilde{M}$ and $m_{k}(q,q)=0$ for all $q\in\tilde{M}$.
\item If \[
c(g,\sigma):=\inf\left\{ k\in\mathbb{R}\,:\, m_{k}(q_{0},q_{1})\in\mathbb{R}\mbox{ for all }q_{0},q_{1}\in\tilde{M}\right\} \]
then $m_{c(g,\sigma)}$ is finite everywhere.
\item We can alternatively define $c(g,\sigma)$ as follows:\emph{\begin{equation}
c(g,\sigma)=\inf_{u\in C^{\infty}(\tilde{M})}\sup_{q\in\tilde{M}}\frac{1}{2}\left|d_{q}u+\theta_{q}\right|^{2}.\label{eq:alt def of c}\end{equation}
}
\end{enumerate}
\end{lem}
We call the number\emph{ }\textbf{\emph{$c(g,\sigma)$}} the \emph{Ma\~né
critical value}. Using \eqref{eq:alt def of c} it is clear that $c(g,\sigma)<\infty$
if and only if $\theta$ is bounded, that is if\begin{equation}
\left\Vert \theta\right\Vert _{\infty}:=\sup_{q\in\tilde{M}}\left|\theta_{q}\right|<\infty.\label{eq:better bound}\end{equation}

\subsection*{The functional $S_{k}$}

$\ $\vspace{6 pt}

We will now define a second functional $S_{k}$, which will be the
main object of study of the present work. In the case $c(g,\sigma)<\infty$,
$S_{k}$ is defined on $\Lambda_{M}\times\mathbb{R}^{+}$. For $c(g,\sigma)=\infty$,
$S_{k}$ is only defined on $\Lambda_{0}\times\mathbb{R}^{+}$. The
following lemma is the key observation required to define $S_{k}$.
In the statement, $\mathbb{T}^{2}$ denotes the $2$-torus.
\begin{lem}
\emph{\label{lem:key observation}}Suppose $c(g,\sigma)<\infty$.
Then for any smooth map $f:\mathbb{T}^{2}\rightarrow M$, $f^{*}\sigma$
is exact.\end{lem}
\begin{proof}
Consider $G:=f_{*}(\pi_{1}(\mathbb{T}^{2}))\leq\pi_{1}(M).$ Then
$G$ is amenable, since $\pi_{1}(\mathbb{T}^{2})=\mathbb{Z}^{2}$,
which is amenable. Then \cite[Lemma 5.3]{Paternain2006} tells us
that since $\left\Vert \theta\right\Vert _{\infty}<\infty$ we can
replace $\theta$ by a $G$-invariant primitive $\theta'$ of $\tilde{\sigma}$,
which descends to define a primitive $\theta''\in\Omega^{1}(\mathbb{T}^{2})$
of $f^{*}\sigma$.
\end{proof}
For each free homotopy class $\nu\in[\mathbb{T},M]$, pick a reference
loop $x_{\nu}\in\Lambda_{\nu}$. Given any $x\in\Lambda_{\nu}$, let
$C(x)$ denote a cylinder with boundary $x(\mathbb{T})\cup x_{\nu}(\mathbb{T})$. 

Define $S_{k}:\Lambda_{\nu}\times\mathbb{R}^{+}\rightarrow\mathbb{R}$
by \[
S_{k}(x,T):=\int_{0}^{1}\frac{1}{2T}\left|\dot{x}(t)\right|^{2}dt+kT-\int_{C(x)}\sigma,\]

This is well defined because the value of $\int_{C(x)}\sigma$ is
\emph{independent of the choice of cylinder}: if $C'(x)$ is another
cylinder with boundary $x(\mathbb{T})\cup x_{\nu}(\mathbb{T})$ then
$\mathbb{T}^{2}(x):=C(x)\cup\overline{C'(x)}$ is a torus (where $\overline{C'(x)}$
denotes the cylinder $C'(x)$ taken with the opposite orientation),
and $\int_{\mathbb{T}^{2}(x)}\sigma=0$ as $\sigma|_{\mathbb{T}^{2}(x)}$
is exact by the previous lemma.\newline

If $c(g,\sigma)=\infty$ then we cannot define $S_{k}$ on all of
$\Lambda_{M}\times\mathbb{R}^{+}$, since in this case Lemma \ref{lem:key observation}
fails. It is however well defined on $\Lambda_{0}\times\mathbb{R}^{+}$.
In order to see why, let us consider the case of contractible loops
when $c(g,\sigma)<\infty$ again. If $x:\mathbb{T}\rightarrow M$
is contractible and $\boldsymbol{x}:D^{2}\rightarrow M$ denotes a
capping disc, so that $\boldsymbol{x}|_{\partial D^{2}}=x$, then
it is easy to see that \begin{equation}
\int_{C(x)}\sigma=\int_{D^{2}}\boldsymbol{x}^{*}\sigma;\label{eq:capping disc}\end{equation}
note that the right-hand side is (as it should be) independent of
the choice of capping disc $\boldsymbol{x}$, and depends only on
$x$ and $\sigma$, since $\sigma|_{\pi_{2}(M)}=0$. Moreover the
right-hand side is well defined and depends only on $x$ and $\sigma$
even when $c(g,\sigma)=\infty$. Thus it makes sense to \emph{define
}$S_{k}|_{\Lambda_{0}\times\mathbb{R}^{+}}$ by \[
S_{k}(x,T)=\int_{0}^{1}\frac{1}{2T}\left|\dot{x}(t)\right|^{2}dt+kT-\int_{D^{2}}\boldsymbol{x}^{*}\sigma;\]
this is consistent with the previous definition of $S_{k}|_{\Lambda_{0}\times\mathbb{R}^{+}}$
when $c(g,\sigma)<\infty$.\newline

Next we will explicitly calculate the derivative of $S_{k}$. Let
$(x_{s},T_{s})$ be a variation of $(x,T)$, with $\xi(t):=\frac{\partial}{\partial s}\bigl|_{s=0}x_{s}(t)$
and $\psi:=\frac{\partial}{\partial s}\bigl|_{s=0}T_{s}$. Write $E_{q}$
and $E_{v}$ for $\frac{\partial E}{\partial q}$ and $\frac{\partial E}{\partial v}$
respectively. Then an easy calculation in local coordinates show that
the first variation (i.e. the Gateaux derivative)\emph{ }of $S_{k}$
at $(\xi,\psi)$, that is, $\frac{\partial}{\partial s}\bigl|_{s=0}S_{k}(x_{s},T_{s})$
is given by:\begin{align}
\frac{\partial}{\partial s}\bigl|_{s=0}S_{k}(x_{s},T_{s}) & =\psi\int_{0}^{1}\left\{ k-E(x(t),\dot{x}(t)/T)\right\} dt+\int_{0}^{1}\sigma_{x(t)}(\xi(t),\dot{x}(t))dt\nonumber \\
 & +\int_{0}^{1}\left\{ T\cdot E_{q}(x(t),\dot{x}(t)/T)\cdot\xi(t)+E_{v}(x(t),\dot{x}(t)/T)\cdot\dot{\xi}(t)\right\} dt.\label{eq:local coordinates}\end{align}
 We claim now that $S_{k}$ is differentiable with respect to the
canonical Hilbert manifold structure of $\Lambda_{\nu}\times\mathbb{R}^{+}$
(i.e. $S_{k}$ is Fréchet differentiable). In fact, $S_{k}$ is of
class $C^{2}$. For this we quote the fact that \[
(x,T)\mapsto\int_{0}^{1}\frac{1}{2T}\left|\dot{x}(t)\right|^{2}dt+kT\]
is of class $C^{2}$ (see for instance \cite{AbbondandoloSchwarz2008})
and thus is remains to check that $x\mapsto\int_{C(x)}\sigma$ is
differentiable. This can be checked directly. It thus follows that
the first variation $\frac{\partial}{\partial s}\bigl|_{s=0}S_{k}(x_{s},T_{s})$
is actually equal to the (Fréchet) derivative $d_{(x,T)}S_{k}(\xi,\psi)$.\newline

Finally let us note that \begin{equation}
\frac{\partial}{\partial T}S_{k}(x,T)=\frac{1}{T}\int_{0}^{T}\{k-2E(y,\dot{y})\}dt,\label{eq:d/dT}\end{equation}
where $y(t):=x(t/T)$.

\subsection*{Relating $S_{k}$ and $A_{k}$}

$\ $\vspace{6 pt}

Next, if $(x,T)$ is a critical point of $S_{k}$ then $y(t):=x(t/T)$
satisfies\[
\int_{0}^{T}\left\{ E_{q}(y,\dot{y})-\frac{d}{dt}E_{v}(y,\dot{y})\right\} \zeta dt-\frac{1}{T}\int_{0}^{T}\sigma_{y}(\zeta,\dot{y})dt=0,\]
where $\zeta(t)=\xi(t/T)$. Since this holds for all variations $\zeta$,
this implies that if $\tilde{y}:[0,T]\rightarrow\tilde{M}$ is a lift
of $y$ then $\tilde{y}$ satisfies the Euler-Lagrange equations for
$L$, that is,\[
L_{q}(\tilde{y},\dot{\tilde{y}})-\frac{d}{dt}L_{v}(\tilde{y},\dot{\tilde{y}})=0.\]
Thus $\tilde{y}$ is the lift to $\tilde{M}$ of the projection to
$M$ of an orbit of $\phi_{t}$, and we have the following result.
\begin{cor}
\label{cor:relating crit points}Let $x\in\Lambda_{M}$ and $\tilde{x}$
denote a lift of $x$ to $\tilde{M}$. Let $T\in\mathbb{R}^{+}$.
Define $\tilde{y}(t):=\tilde{x}(t/T)$. Then the following are equivalent:
\begin{enumerate}
\item The pair $(x,T)$ is a critical point of $S_{k}$,
\item $\tilde{y}$ is a critical point of $A_{k}$.
\end{enumerate}
Thus the pair $(x,T)\in\Lambda_{M}\times\mathbb{R}^{+}$ is a critical
point of $S_{k}$ if and only if $t\mapsto x(t/T)$ is the projection
to $M$ of a closed orbit of $\phi_{t}$.
\end{cor}
In order to specify the lifts we work with, let us fix a lift $\tilde{x}_{\nu}:I\rightarrow\tilde{M}$
of $x_{\nu}$ for each $\nu\in[\mathbb{T},M]$. We remark here that
throughout the paper given any two paths $y,y'$ such that the end
point of $y$ is the start point of $y'$, the path $y*y'$ is the
path obtained by first going along $y$ and then going along $y'$.
Similarly the path $y^{-1}$ is the path obtained by going along $y$
backwards.\newline 

Suppose now that $c(g,\sigma)<\infty$. Fix a free homotopy class
$\nu\in[\mathbb{T},M]$ (which could be the trivial free homotopy
class). Let $x\in\Lambda_{\nu}$, and let $x_{s}$ denote a free homotopy
from $x_{0}=x$ to $x_{1}=x_{\nu}$. Let $z(s):=x_{s}(0)$. Let $\tilde{x}_{s}$
denote the unique homotopy of curves on $\tilde{M}$ that projects
down onto $x_{s}$ and satisfies with $\tilde{x}_{1}(t)=\tilde{x}_{\nu}(t)$.
Let $\tilde{x}(t):=\tilde{x}_{0}(t)$, $\tilde{z}_{0}(s):=\tilde{x}_{s}(0)$
and $\tilde{z}_{1}(s):=\tilde{x}_{s}(1)$. 

Now observe that if $R\subseteq\tilde{M}$ denotes the rectangle $R=\mbox{im}\,\tilde{x}_{s}$
then we have \[
\int_{C(x)}\sigma=\int_{R}\tilde{\sigma}=\int_{R}d\theta=\int_{\partial R}\theta=\int_{\tilde{x}*\tilde{z}_{1}*\tilde{x}_{\nu}^{-1}*\tilde{z}_{0}^{-1}}\theta.\]
Let $\varphi\in\pi_{1}(M)$ denotes the unique covering transformation
taking $\tilde{z}_{0}$ to $\tilde{z}_{1}$. Since $\left\langle \varphi\right\rangle \leq\pi_{1}(M)$
is an amenable subgroup, \cite[Lemma 5.3]{Paternain2006} allows us
to assume that without loss of generality, $\theta$ is $\varphi$-invariant.
Thus \[
\int_{\tilde{z}_{0}^{-1}}\theta+\int_{\tilde{z}_{1}}\theta=0.\]
It thus follows that \[
\int_{C(x)}\sigma=\int_{\tilde{x}}\theta+\int_{\tilde{x}_{\nu}^{-1}}\theta.\]

Set \[
a_{\nu}:=\int_{\tilde{x}_{\nu}^{-1}}\theta.\]
We conclude: \[
\int_{C(x)}\sigma=\int_{\tilde{x}}\theta+a_{\nu}.\]
This computation shows if $c(g,\sigma)<\infty$ then for any $(x,T)\in\Lambda_{\nu}\times\mathbb{R}^{+}$,
if $\tilde{x}$ is any lift of $x$ and $\tilde{y}(t):=\tilde{x}(t/T)$
then we have \begin{equation}
S_{k}(x,T)=A_{k}(\tilde{y})+a_{\nu}.\label{eq:Sk and Ak}\end{equation}
For the case $\nu=0\in[\mathbb{T},M]$ the trivial free homotopy class,
we may choose above the curve $x_{0}$ to be a constant map, from
which it is easy to see that $a_{0}=0$. In particular, if $(x,T)\in\Lambda_{0}\times\mathbb{R}^{+}$
and $\tilde{y}$ is defined as before then \begin{equation}
S_{k}(x,T)=A_{k}(\tilde{y}).\label{eq:contractibl}\end{equation}
Finally, if $c(g,\sigma)=\infty$, $S_{k}$ is only defined on $\Lambda_{0}\times\mathbb{R}^{+}$,
and it is clear that \eqref{eq:contractibl} still holds.

\section{The Palais-Smale condition }

Let $(\mathcal{M},\left\langle \cdot,\cdot\right\rangle )$ be a Riemannian
Hilbert manifold, and let $S:\mathcal{M}\rightarrow\mathbb{R}$ be
of class $C^{1}$. 
\begin{defn}
We say that $S$ satisfies the \emph{Palais-Smale condition }if every
sequence $(x_{n})\subseteq\mathcal{M}$ such that $\left\Vert d_{x_{n}}S\right\Vert \rightarrow0$
as $n\rightarrow\infty$ and $\sup_{n}S(x_{n})<\infty$ admits a convergent
subsequence. We say that $S$ satisfies the \emph{Palais-Smale condition
at the level $\mu\in\mathbb{R}$} if every sequence $(x_{n})\subseteq\mathcal{M}$
with $\left\Vert d_{x_{n}}S\right\Vert \rightarrow0$ as $n\rightarrow\infty$
and $S(x_{n})\rightarrow\mu$ admits a convergent subsequence.
\end{defn}
In this section we will prove the following result about when the
functional $S_{k}$ satisfies the Palais-Smale condition. This result
is adapted from \cite[Proposition 3.8, Proposition 3.12]{Contreras2006}
We will first consider only the case where $c(g,\sigma)<\infty$ (see
Proposition \ref{thm:(the-key-theorem new)} below for the case $c(g,\sigma)=\infty$).
In the statement of the theorem, $\left\Vert \cdot\right\Vert $ denotes
the operator norm with respect to the metric $\left\langle \cdot,\cdot\right\rangle $.
\begin{thm}
\label{thm:the Palais Smale theorem}Suppose $c(g,\sigma)<\infty$.
Let $A,B,k\in\mathbb{R}^{+}$, and suppose $(x_{n},T_{n})\subseteq\Lambda_{M}\times\mathbb{R}^{+}$
satisfies:\[
\sup_{n}\left|S_{k}(x_{n},T_{n})\right|\leq A,\ \ \ \sup_{n}T_{n}\leq B,\ \ \ \left\Vert d_{(x_{n},T_{n})}S_{k}\right\Vert <\frac{1}{n}.\]
Then if:
\begin{enumerate}
\item If $\liminf T_{n}>0$ then passing to a subsequence if necessary the
sequence $(x_{n},T_{n})$ is convergent in the $W^{1,2}$-topology.
\item If $\liminf T_{n}=0$ and the $x_{n}$ are all contractible, passing
to a subsequence if necessary it holds that $S_{k}(x_{n},T_{n})\rightarrow0$.
\end{enumerate}
\end{thm}
Before proving the theorem, let us now fix some notation that we will
use throughout this section, as well as implicitly throughout the
rest of the paper. Given a sequence $(x_{n},T_{n})\subseteq\Lambda_{M}\times\mathbb{R}^{+}$,
let $y_{n}:[0,T_{n}]\rightarrow M$ be defined by $y_{n}(t):=x_{n}(t/T_{n})$.
Define:\[
l_{n}:=\int_{0}^{1}\left|\dot{x}_{n}(t)\right|dt;\]
 \[
e_{n}:=\int_{0}^{1}\frac{1}{2T_{n}}\left|\dot{x}_{n}(t)\right|^{2}dt.\]
Note that $l_{n}$ is the length of $y_{n}$ and $e_{n}$ is the energy
of $y_{n}$. The Cauchy-Schwarz inequality implies \begin{equation}
l_{n}^{2}\leq2T_{n}e_{n}.\label{eq:cs ineq}\end{equation}

Suppose now $c(g,\sigma)<\infty$. Since $\left\Vert \theta\right\Vert _{\infty}<\infty$,
there exist constants $b_{1},b_{2}\in\mathbb{R}^{+}$ such that \begin{equation}
L(q,v)\geq b_{1}\left|v\right|^{2}-b_{2}\label{eq:superlinearity}\end{equation}
for all $(q,v)\in T\tilde{M}$. 

Given $A,B,k\in\mathbb{R}^{+}$ and a free homotopy class $\nu\in[\mathbb{T},M]$,
let $\mathbb{D}(A,B,k,\nu)\subseteq\Lambda_{M}\times\mathbb{R}^{+}$
denote the set of pairs $(x,T)$ such that $x\in\Lambda_{\nu}$, $S_{k}(x,T)\leq A$
and $T\leq B$.

\subsection*{Proof of Theorem \ref{thm:the Palais Smale theorem}}

$\ $\vspace{6 pt}

We begin with three preparatory lemmata.
\begin{lem}
\label{lem:energy bound}Suppose $c(g,\sigma)<\infty$. Let $(x_{n},T_{n})\subseteq\mathbb{D}(A,B,k,\nu)$.
Then if \[
b(A,B,\nu):=\frac{A+b_{2}B+\left|a_{\nu}\right|}{2b_{1}}\]
 it holds that \[
e_{n}\leq b(A,B,\nu)\]
for all $n\in\mathbb{N}$.\end{lem}
\begin{proof}
We have by \eqref{eq:Sk and Ak} and \eqref{eq:superlinearity} that

\begin{align*}
A & \geq S_{k}(x_{n},T_{n})\\
 & =A_{k}(\tilde{y}_{n})-a_{\nu}\\
 & \geq2b_{1}e_{n}-b_{2}T_{n}+kT_{n}-a_{\nu},\end{align*}
and thus \[
e_{n}\leq\frac{A+b_{2}T_{n}-kT_{n}+\left|a_{\nu}\right|}{2b_{1}}\leq\frac{A+b_{2}B+\left|a_{\nu}\right|}{2b_{1}}.\]
\end{proof}
\begin{lem}
\label{lem:Cun goes to zero}Suppose $c(g,\sigma)<\infty$, and suppose
$(x_{n})\subseteq\Lambda_{0}$ are such that $l_{n}\rightarrow0$.
Then $\int_{C(x_{n})}\sigma\rightarrow0$.\end{lem}
\begin{proof}
Let $\boldsymbol{x}_{n}:D^{2}\rightarrow M$ denote a capping disc
for $x_{n}$, so $\boldsymbol{x}_{n}|_{\partial D^{2}}=x_{n}$ and
$\int_{C(x_{n})}\sigma=\int_{D^{2}}\boldsymbol{x}_{n}^{*}\sigma$
(as in \eqref{eq:capping disc}).

Let $\tilde{\boldsymbol{x}}_{n}:D^{2}\rightarrow\tilde{M}$ denote
a lift of $\boldsymbol{x}_{n}$ to $\tilde{M}$. Then \[
\left|\int_{D^{2}}\boldsymbol{x}_{n}^{*}\sigma\right|=\left|\int_{D^{2}}\tilde{\boldsymbol{x}}_{n}^{*}(d\theta)\right|=\left|\int_{\tilde{x}_{n}}\theta\right|\leq\left\Vert \theta\right\Vert _{\infty}l_{n}\rightarrow0.\]

\end{proof}
We now reduce the first statement of Theorem \ref{thm:the Palais Smale theorem}
to the following simpler situation:
\begin{lem}
\label{pro:special case of PS}Suppose $c(g,\sigma)<\infty$ and $(x_{n},T_{n})\in\mathbb{D}(A,B,k,\nu)$
with $\liminf T_{n}>0$. Passing to a subsequence we may assume that
there exists $x\in\Lambda_{\nu}$ such that the $x_{n}$ converge
to $x$ in the $C^{0}$-topology.\end{lem}
\begin{proof}
Firstly by compactness of $M$, passing to a subsequence if necessary
we may assume there exists $q\in M$ and $T\in\mathbb{R}^{+}$ such
that $\lim_{n\rightarrow\infty}x_{n}(0)=x_{n}(1)=q$ and $\lim_{n\rightarrow\infty}T_{n}=T$.
Consider $g$-geodesics $c_{n}:I\rightarrow M$ such that $c_{n}(0)=q$
and $c_{n}(1)=x_{n}(0)$. By passing to a subsequence we may assume
that $\mbox{dist}_{g}(x_{n}(0),q)<1$, and thus we have $\left|\dot{c}_{n}\right|\leq1$.
Now consider the curves $w_{n}:[0,T_{n}+2]\rightarrow M$ defined
by \[
w_{n}(t)=c_{n}*y_{n}*c_{n}^{-1},\]
and $z_{n}:\mathbb{T}\rightarrow M$ defined by \[
z_{n}(t)=w_{n}(t/T_{n}+2).\]
Thus $z_{n}(0)=z_{n}(1)=q$, and $(z_{n})\subseteq\Lambda_{\nu}$.

Given $0\leq t_{1}<t_{2}<T_{n}+2$, \[
\mbox{dist}_{g}(w_{n}(t_{1}),w_{n}(t_{2}))\leq\int_{t_{1}}^{t_{2}}\left|\dot{w}_{n}(t)\right|dt\leq\sqrt{2}\left|t_{2}-t_{1}\right|^{1/2}\left(\int_{0}^{T_{n}+2}\frac{1}{2}\left|\dot{w}_{n}(t)\right|^{2}dt\right)^{1/2}.\]
By Lemma \ref{lem:energy bound} we have\begin{align*}
\int_{0}^{T_{n}+2}\frac{1}{2}\left|\dot{w}_{n}(t)\right|^{2}dt & =\int_{0}^{1}\frac{1}{2}\left|\dot{c}_{n}(t)\right|^{2}dt+e_{n}+\int_{0}^{1}\frac{1}{2}\left|\dot{c}_{n}^{-1}(t)\right|^{2}dt\\
 & \leq1+b(A,B,\nu),\end{align*}
and thus\[
\mbox{dist}_{g}(w_{n}(t_{1}),w_{n}(t_{2}))\leq\sqrt{2}\left|t_{2}-t_{1}\right|^{1/2}(1+b(A,B,\nu))^{1/2}.\]
Hence the family $(w_{n})$ is equicontinuous. The Arzelá-Ascoli theorem
then completes the proof. 
\end{proof}
We will now prove Theorem \ref{thm:the Palais Smale theorem}.
\begin{proof}
\emph{(of Theorem \ref{thm:the Palais Smale theorem})}

We begin by proving the first statement of the theorem. This part
of the proof is very similar to the proof of \cite[Theorem B]{ContrerasIturriagaPaternainPaternain2000}.
Suppose $(x_{n},T_{n})\subseteq\mathbb{D}(A,B,k,\nu)$ with $\liminf T_{n}>0$.
By the previous lemma, after passing to a subsequence if necessary,
we may assume that $(x_{n},T_{n})$ converges in the $C^{0}$-topology
to some $(x,T)$ where $T>0$. 

Without loss of generality, let us assume that the limit curve $x$
is contained in a single chart $U$ (otherwise simply repeat these
arguments finitely many times). Then after passing possibly to another
subsequence, we may assume that the $x_{n}$ are all contained in
$U$ as well. There exists a constant $b_{3}\in\mathbb{R}^{+}$ such
that in the coordinates on $U$, \begin{equation}
b_{3}:=\sup_{q\in U,v\in T_{q}M}\frac{\left|E_{q}(q,v)\right|}{1+\left|v\right|^{2}}<\infty.\label{eq:b3}\end{equation}

Write $z_{n}(t):=\frac{1}{T_{n}}x_{n}(t)$. By Lemma \ref{lem:energy bound}
we can find a constant $R>0$ such that \[
\left|x_{n}\right|_{1,2}\leq R,\ \ \ \left|z_{n}\right|_{1,2}\leq R.\]
Now since $\left\Vert d_{(x_{n},T_{n})}S_{k}\right\Vert \rightarrow0$
as $n\rightarrow\infty$, given $\varepsilon>0$ there exists $N\in\mathbb{N}$
such that for every $(\xi,\psi)$ satisfying $\left|(\xi,\psi)\right|\leq2R$
and $n,m\geq N$ we have\[
\left|d_{(x_{n},T_{n})}S_{k}(\xi,\psi)-d_{(x_{m},T_{m})}S_{k}(\xi,\psi)\right|<\varepsilon.\]
Take $\xi=x_{n}-x_{m}$ and $\psi=0$ and use \eqref{eq:local coordinates}
to discover:\begin{multline}
\left|\int_{0}^{1}\left\{ T_{n}\cdot E_{q}(x_{n},\dot{z}_{n})-T_{m}\cdot E_{q}(x_{m},\dot{z}_{m})\right\} (x_{n}-x_{m})dt\right.\\
+\int_{0}^{1}\left\{ E_{v}(x_{n},\dot{z}_{n})-E_{v}(x_{m},\dot{z}_{m})\right\} (\dot{x}_{n}-\dot{x}_{m})dt\\
\left.+\int_{0}^{1}\sigma_{x_{n}}(\dot{x}_{n},\dot{x}_{m})-\sigma_{x_{m}}(\dot{x}_{n},\dot{x}_{m})dt\right|<\varepsilon.\label{eq:three integrals}\end{multline}
Here we are using the canonical parallel transport available to us
on Euclidean spaces to view $\dot{x}_{n}-\dot{x}_{m}$ as a tangent
vector in any tangent space of our choosing. Using \eqref{eq:b3}
we can bound the first integral as follows: \[
\left|\int_{0}^{1}\left\{ T_{n}\cdot E_{q}(x_{n},\dot{z}_{n})-T_{m}\cdot E_{q}(x_{m},\dot{z}_{m})\right\} (x_{n}-x_{m})dt\right|\leq(2Bb_{3}+2b_{3}b(A,B,\nu))\left\Vert x_{n}-x_{m}\right\Vert _{\infty}.\]
Let us write $\sigma|_{U}$ in local coordinates as $\sigma=\sigma_{ij}dq^{i}\wedge dq^{j}$,
where $\sigma_{ij}\in C^{\infty}(U,\mathbb{R})$. Then since\[
\left|\sigma_{ij}(x_{n}(t))-\sigma_{ij}(x_{m}(t))\right|\rightarrow0\ \ \ \mbox{as }n,m\rightarrow\infty,\mbox{\ uniformly in }t,\]
and\[
\int_{0}^{1}\left|\dot{x}_{n}\right|\left|\dot{x}_{m}\right|dt\leq2\sqrt{T_{n}T_{m}e_{n}e_{m}}\]
is bounded, it follows that for $n,m$ large the third integral is
small. Thus the second integral must also be small for large $n,m$.
Since \[
\left|v-v'\right|^{2}=\left(E_{v}(q,v)-E_{v}(q',v')\right)\cdot(v-v'),\]
we have \[
\int_{0}^{1}\left|\dot{z}_{n}-\dot{z}_{m}\right|^{2}dt\leq\int_{0}^{1}\left\{ E_{v}(x_{n},\dot{z}_{n})-E_{v}(x_{m},\dot{z}_{m})\right\} (\dot{z}_{n}-\dot{z}_{m})dt,\]
and hence the fact that the second integral in \eqref{eq:three integrals}
is small for large $n,m$ implies that the sequence $(z_{n})$, and
hence the sequence $(x_{n})$, converges in the $W^{1,2}$-topology.
This completes the proof of the first statement of Theorem \ref{thm:the Palais Smale theorem}.\newline

We now prove the second statement of Theorem \ref{thm:the Palais Smale theorem}.
This part of the proof follows the proof of \cite[Theorem 3.8]{Contreras2006}
very closely. Assume $(x_{n},T_{n})\subseteq\mathbb{D}(A,B,k,0)$
(where $0\in[\mathbb{T},M]$ denotes the trivial free homotopy class)
and that $\liminf T_{n}=0$. Passing to a subsequence we may assume
that $T_{n}\rightarrow0$. It suffices to show that passing to a subsequence
we have $e_{n}\rightarrow0$. Then \[
S_{k}(x_{n},T_{n})=e_{n}+kT_{n}-\int_{C(x_{n})}\sigma\rightarrow0\]
by Lemma \ref{lem:Cun goes to zero}.

We know that $e_{n}$ remains bounded by Lemma \ref{lem:energy bound}.
Since $T_{n}\rightarrow0$, \eqref{eq:cs ineq} implies that $l_{n}\rightarrow0$.
Thus as before we may assume that all the curves $x_{n}$ take their
image in the domain of some chart $U$ on $M$. Thus for the remainder
of the proof we work in coordinates as if $M=\mathbb{R}^{\textrm{dim}\, M}$.
Let $\xi_{n}(t):=x_{n}(t)-x_{n}(0)$ so that $\xi_{n}(0)=\xi_{n}(1)=0$.
Then $(\xi_{n},0)\in T_{(x_{n},T_{n})}(\Lambda_{\mathbb{R}^{\textrm{dim}\, M}}\times\mathbb{R}^{+})$.
Let also $\zeta_{n}(t):=\xi_{n}(t/T_{n})$, so that $\dot{\zeta}_{n}(t)=\dot{y}_{n}(t)$.
Then \begin{align*}
\left|d_{(x_{n},T_{n})}S_{k}(\xi_{n},0)\right| & \leq\frac{1}{n}\left(T_{n}\int_{0}^{T_{n}}\left|\dot{\zeta}_{n}(t)\right|^{2}dt\right)^{1/2}\\
 & \leq\frac{1}{n}\sqrt{2T_{n}e_{n}}.\end{align*}
 Using \eqref{eq:local coordinates} we have \begin{align*}
d_{(x_{n},T_{n})}S_{k}(\xi_{n},0) & =\int_{0}^{T_{n}}\left\{ E_{q}(y_{n},\dot{y}_{n})\cdot\zeta_{n}+E_{v}(y_{n},\dot{y}_{n})\cdot\dot{\zeta}_{n}(t)\right\} dt\\
 & +\int_{0}^{1}\sigma_{x_{n}(t)}(\xi_{n}(t),\dot{x}_{n}(t))dt.\end{align*}
There exists $b_{4}\in\mathbb{R}^{+}$ such that \[
\left|\int_{0}^{1}\sigma_{x_{n}(t)}(\xi_{n}(t),\dot{x}_{n}(t))dt\right|\leq b_{4}\int_{0}^{1}\left|\xi_{n}(t)\right|\left|\dot{x}_{n}(t)\right|dt\leq b_{4}l_{n}^{2}.\]
Thus using \eqref{eq:b3} and the fact that \[
E_{v}(q,v)\cdot\xi=\left\langle v,\xi\right\rangle \]
we have \begin{align*}
d_{(x_{n},T_{n})}S_{k}(\xi_{n},0) & \geq-b_{3}\int_{0}^{T_{n}}\left(1+\left|\dot{y}_{n}(t)\right|^{2}\right)\left|y_{n}(t)-y_{n}(0)\right|dt+2e_{n}-b_{4}l_{n}^{2}\\
 & \geq-b_{3}l_{n}(T_{n}+2e_{n})+2e_{n}-b_{4}l_{n}^{2}.\end{align*}
Putting this together and dividing through by $\sqrt{T_{n}}$ we have
\[
-b_{3}l_{n}\sqrt{T_{n}}-2b_{3}\frac{e_{n}l_{n}}{\sqrt{T_{n}}}+2\frac{e_{n}}{\sqrt{T_{n}}}-b_{4}\frac{l_{n}^{2}}{\sqrt{T_{n}}}\leq\frac{1}{n}\sqrt{2e_{n}}.\]
By \eqref{eq:cs ineq}, we have\[
\lim_{n\rightarrow\infty}\frac{l_{n}^{2}}{\sqrt{T_{n}}}=0,\ \ \ \lim_{n\rightarrow\infty}\frac{l_{n}}{\sqrt{T_{n}}}\ \ \ \mbox{bounded};\]
thus we must have \[
\lim_{n\rightarrow\infty}\frac{e_{n}}{\sqrt{T_{n}}}\ \ \ \mbox{bounded,}\]
and this can happen if and only if $e_{n}\rightarrow0$. This completes
the proof of the second statement of Theorem \ref{thm:the Palais Smale theorem}.
\end{proof}
We now wish to study the case where $c(g,\sigma)=\infty$. Recall
in this case $S_{k}$ is only defined on $\Lambda_{0}\times\mathbb{R}^{+}$.
In order for a result similar to the above theorem to hold in the
unbounded setting, we must restrict to a subset of $\Lambda_{0}\times\mathbb{R}^{+}$. 
\begin{defn}
Suppose $K\subseteq\tilde{M}$ is compact. Define $\Lambda_{0}^{K}\subseteq\Lambda_{0}$
to be the set of loops $x\in\Lambda_{0}$ such that there exists a
lift $\tilde{x}:\mathbb{T}\rightarrow\tilde{M}$ of $x$ such that
$\tilde{x}(\mathbb{T})\subseteq K$. 
\end{defn}
Here is the extension of Theorem \ref{thm:the Palais Smale theorem}
to the case $c(g,\sigma)=\infty$.
\begin{prop}
\textbf{\textup{\label{thm:(the-key-theorem new)}}}Suppose that $c(g,\sigma)=\infty$.
Let $A,B,k\in\mathbb{R}^{+}$ and take $K\subseteq\tilde{M}$ compact.
Suppose $(x_{n},T_{n})\subseteq\Lambda_{0}^{K}\times\mathbb{R}^{+}$
satisfy:\[
\sup_{n}\left|S_{k}(x_{n},T_{n})\right|\leq A,\ \ \ \sup_{n}T_{n}\leq B,\ \ \ \left\Vert d_{(x_{n},T_{n})}S_{k}\right\Vert <\frac{1}{n}.\]
Then:
\begin{enumerate}
\item If $\liminf T_{n}>0$ then passing to a subsequence if necessary the
sequence $(x_{n},T_{n})$ is convergent in the $W^{1,2}$-topology.
\item If $\liminf T_{n}=0$, passing to a subsequence if necessary it holds
that $S_{k}(x_{n},T_{n})\rightarrow0$.
\end{enumerate}
\end{prop}
\begin{proof}
The proof proceeds exactly as before, since any primitive $\theta$
of $\tilde{\sigma}$ is necessarily bounded on $K$.
\end{proof}

\section{Supercritical energy levels: the case $k>c(g,\sigma)$}

In this section we assume $c(g,\sigma)<\infty$, and study \emph{supercritical
energies} $k>c(g,\sigma)$. We aim to prove the first statement of
Theorem \ref{thm:theorem A}. The key fact we will use is the following
result. As before, let $\left(\mathcal{M},\left\langle \cdot,\cdot\right\rangle \right)$
be a Riemannian Hilbert manifold, and let $S:\mathcal{M}\rightarrow\mathbb{R}$
be of class $C^{1}$.
\begin{prop}
\label{thm:morse}Suppose $S$ is bounded below, satisfies the Palais-Smale
condition and for every $A\in\mathbb{R}^{+}$ the set $\left\{ x\in\mathcal{M}\,:\, S(x)\leq A\right\} $
is complete. Then $S$ has a global minimum.
\end{prop}
A proof may be found in \cite[Corollary 23]{ContrerasIturriagaPaternainPaternain2000}.
Fix a non-trivial free homotopy class $\nu\in[\mathbb{T},M]$. The
aim of this section is to verify that for $k>c(g,\sigma)$, the functional
$S_{k}$ on the Hilbert manifold $\Lambda_{\nu}\times\mathbb{R}^{+}$
satisfies the hypotheses of Proposition \ref{thm:morse}. For then
the global minimum whose existence Theorem \ref{thm:morse} guarantees
is our desired closed orbit of energy $k$. 

The first step then is the following lemma, whose proof only requires
$k\geq c(g,\sigma)$, and works for any free homotopy class $\nu\in[\mathbb{T},M]$.
\begin{lem}
Let $k\geq c(g,\sigma)$. Then $S_{k}|_{\Lambda_{\nu}\times\mathbb{R}^{+}}$
is bounded below.\emph{ }\end{lem}
\begin{proof}
The argument begins by replicating an argument seen earlier in Section
\ref{sec:Preliminaries}. Fix a free homotopy class $\nu\in[\mathbb{T},M]$
(which could be the trivial free homotopy class). Let $(x,T)\in\Lambda_{\nu}\times\mathbb{R}^{+}$,
and let $x_{s}$ denote a free homotopy from $x_{0}=x$ to $x_{1}=x_{\nu}$.
Let $z(s):=x_{s}(0)$. Lift $x_{s}$ to a homotopy $\tilde{x}_{s}$
in $\tilde{M}$ with $\tilde{x}_{1}(t)=\tilde{x}_{\nu}(t)$, and let
$\tilde{x}(t):=\tilde{x}_{0}(t)$, $\tilde{z}_{0}(s)=\tilde{x}_{s}(0)$
and $\tilde{z}_{1}(s)=\tilde{x}_{s}(1)$. 

Now observe that if $R\subseteq\tilde{M}$ denotes the rectangle $R=\mbox{im}\,\tilde{x}_{s}$
then we have \[
\int_{C(x)}\sigma=\int_{R}\tilde{\sigma}=\int_{R}d\theta=\int_{\partial R}\theta=\int_{\tilde{x}*\tilde{z}_{1}*\tilde{x}_{\nu}^{-1}*\tilde{z}_{0}^{-1}}\theta.\]

Suppose $\varphi\in\pi_{1}(M)$ denotes the unique covering transformation
taking $\tilde{z}_{0}$ to $\tilde{z}_{1}$. Since $\left\langle \varphi\right\rangle \leq\pi_{1}(M)$
is an amenable subgroup, \cite[Lemma 5.3]{Paternain2006} allows us
to assume that without loss of generality, $\theta$ is $\varphi$-invariant.
Thus \[
\int_{\tilde{z}_{0}^{-1}}\theta+\int_{\tilde{z}_{1}}\theta=0.\]
It thus follows that \begin{equation}
\int_{C(x)}\sigma=\int_{\tilde{x}}\theta+\int_{\tilde{x}_{\nu}^{-1}}\theta.\label{eq:C(u) eqn}\end{equation}

Let $\tilde{x}_{n}:=\varphi^{n}\tilde{x}$, and use similar notations
for $\tilde{z}_{n}$ and $\tilde{x}_{\nu,n}$. Let $\tilde{y}_{n}:=\tilde{x}_{n}(t/T)$,
so $\tilde{y}_{n}:[0,T]\rightarrow\tilde{M}$. Then for any $n\in\mathbb{N}$
we can consider the closed loop $u_{n}:[0,T_{n}]\rightarrow\tilde{M}$
defined by \[
u_{n}=\tilde{y}_{0}*\tilde{y}_{1}*\dots*\tilde{y}_{n}*\tilde{z}_{n+1}*\tilde{x}_{\nu,n}^{-1}*\dots*\tilde{x}_{\nu,1}^{-1}*\tilde{x}_{\nu}^{-1}*\tilde{z}_{0}^{-1},\]
where \[
T_{n}:=(n+1)T+1+(n+1)+1.\]
We have \begin{align*}
A_{k}(u_{n}) & =(n+1)\left\{ \int_{0}^{T}\frac{1}{2}\left|\dot{\tilde{y}}(t)\right|^{2}dt+\int_{0}^{1}\frac{1}{2}\left|\dot{\tilde{x}}_{\nu}^{-1}\right|^{2}dt-\int_{\tilde{y}_{0}}\theta-\int_{\tilde{x}_{\nu}^{-1}}\theta\right\} \\
 & +\int_{0}^{1}\frac{1}{2}\left|\dot{\tilde{z}}_{1}(t)\right|^{2}dt+\int_{0}^{1}\frac{1}{2}\left|\dot{\tilde{z}}_{0}^{-1}(t)\right|^{2}dt+kT_{n}.\end{align*}
Now if $k\geq c(g,\sigma)$ then by definition of $c(g,\sigma)$ we
have $A_{k}(u_{n})\geq0$. We thus obtain\begin{align*}
0 & \leq\int_{0}^{T}\frac{1}{2}\left|\dot{\tilde{y}}_{0}(t)\right|^{2}dt+\int_{0}^{1}\frac{1}{2}\left|\dot{\tilde{x}}_{\nu}^{-1}\right|^{2}dt-\int_{\tilde{y}_{0}}\theta-\int_{\tilde{x}_{\nu}^{-1}}\theta+\frac{kT_{n}}{n+1}\\
 & +\frac{1}{n+1}\left(\int_{0}^{1}\left|\dot{\tilde{z}}_{1}(t)\right|^{2}dt+\int_{0}^{1}\left|\dot{\tilde{z}}_{0}^{-1}(t)\right|^{2}dt\right).\end{align*}
Letting $n\rightarrow\infty$ and substituting for the terms with
$\tilde{y}_{0}$ we obtain \begin{equation}
\int_{0}^{1}\frac{1}{2T}\left|\dot{x}(t)\right|^{2}dt+\int_{0}^{1}\frac{1}{2}\left|\dot{x}_{\nu}^{-1}\right|^{2}dt-\int_{\tilde{x}}\theta-a_{\nu}+k(T+1)\geq0.\label{eq:bigger than 0 eqn}\end{equation}
Now \begin{align*}
S_{k}(x,T) & =\int_{0}^{1}\frac{1}{2T}\left|\dot{x}(t)\right|^{2}dt+kT-\int_{C(x)}\sigma\\
 & =\int_{0}^{1}\frac{1}{2T}\left|\dot{x}(t)\right|^{2}dt+kT-\int_{\tilde{x}}\theta-a_{\nu},\end{align*}
and hence by \eqref{eq:C(u) eqn} and \eqref{eq:bigger than 0 eqn},
\[
S_{k}(x,T)+\int_{0}^{1}\frac{1}{2}\left|\dot{\tilde{x}}_{\nu}(t)\right|^{2}dt+k\geq0,\]
that is, \[
S_{k}(x,T)\geq-\int_{0}^{1}\frac{1}{2}\left|\dot{\tilde{x}}_{\nu}(t)\right|^{2}dt-k>-\infty,\]
which completes the proof.
\end{proof}
Let us set\[
i_{k,\nu}:=\inf_{(x,T)\in\Lambda_{\nu}\times\mathbb{R}^{+}}S_{k}(x,T),\]
so that the lemma tells us $i_{k,\nu}>-\infty$ for $k\geq c(g,\sigma)$.\newline

The next lemma implies that $\{S_{k}|_{\Lambda_{\nu}\times\mathbb{R}^{+}}\leq A\}$
is complete for any $A\geq0$.
\begin{lem}
\label{lem:precompactness below-1}Suppose $c(g,\sigma)<\infty$.
Let $\nu\in[\mathbb{T},M]$ be a non-trivial free homotopy class and
$A\in\mathbb{R}^{+}$. There exists $T_{0}=T_{0}(A,k,\nu)\in\mathbb{R}^{+}$
such that if $(x,T)\in\mathbb{D}(A,\infty,k,\nu)$ then $T\geq T_{0}$.\end{lem}
\begin{proof}
Let $\tilde{x}$ denote an admissable lift of $x$ and let $\tilde{y}:[0,T]\rightarrow\tilde{M}$
be the curve $t\mapsto\tilde{x}(t/T)$. Using \eqref{eq:Sk and Ak}
and \eqref{eq:superlinearity} we compute that \begin{align*}
A & \geq S_{k}(x,T)\\
 & =A_{k}(\tilde{y})+a_{\nu}\\
 & \geq\frac{b_{1}}{T}\int_{0}^{1}\left|\dot{\tilde{x}}\right|^{2}dt-(k-b_{2})T+a_{\nu}\\
 & \geq\frac{b_{1}}{T}l(\nu)-(k-b_{2})T+a_{\nu},\end{align*}
where \[
l(\nu):=\inf\left\{ \int_{0}^{1}\left|\dot{x}(t)\right|dt\,:\, x\in\Lambda_{\nu}\right\} .\]
Since $M$ is closed and $\nu$ is a non-trivial free homotopy class,
we have $l(\nu)>0$, which implies the thesis of the lemma. 
\end{proof}

\subsection*{Proof of the first statement of Theorem \ref{thm:theorem A}}

$\ $\vspace{6 pt}

Take $k>c(g,\sigma)$, and fix is a non-trivial free homotopy class
$\nu\in[\mathbb{T},M]$. Let $(x_{n},T_{n})\subseteq\mathbb{D}(A,\infty,k,\nu)$
. We want to show that $(x_{n},T_{n})$ admits a convergent subsequence
in the $W^{1,2}$-topology. In view of Theorem \ref{thm:the Palais Smale theorem},
it suffices to show that there exists $B>0$ such that $\mathbb{D}(A,B,k,\nu)$
and that $\liminf T_{n}>0$. 
\begin{lem}
\label{lem:tn bounded away from zero}The sequence $(T_{n})$ is bounded
above and bounded away from zero.\end{lem}
\begin{proof}
First we claim that $(T_{n})$ is bounded. Indeed, if $c=c(g,\sigma)$,
\begin{align*}
A & \geq S_{k}(x_{n},T_{n})\\
 & =S_{c}(x_{n},T_{n})+(k-c)T_{n}\\
 & \geq i_{c,\nu}+(k-c)T_{n},\end{align*}
and thus $(T_{n})$ is bounded, say $T_{n}\leq B$ for all $n$, where
$B\in\mathbb{\mathbb{R}}^{+}$. Passing to a subsequence we may assume
that if $T:=\liminf T_{n}$ then $T_{n}\rightarrow T$. It remains
to check $T>0$. From \eqref{eq:cs ineq} and Lemma \ref{lem:energy bound}
if $T=0$ then $l_{n}\rightarrow0$. But this is a contradiction as
$l_{n}>l(\nu)>0$ (see the proof of the previous lemma).
\end{proof}

\section{Subcritical energy levels: the case $k<c(g,\sigma)$}

In this section we drop the assumption that $c(g,\sigma)<\infty$,
and study \emph{subcritical energies} $k<c(g,\sigma)$.

\subsection*{Mountain pass geometry}

$\ $\vspace{6 pt}

As before, let $(\mathcal{M},\left\langle \cdot,\cdot\right\rangle )$
be a Riemannian Hilbert manifold and $S:\mathcal{M}\rightarrow\mathbb{R}$
a function of class $C^{2}$. Let $\Phi_{s}$ denote the (local) flow
of $-\nabla S$. Define $\alpha:\mathcal{M}\rightarrow\mathbb{R}^{+}\cup\{\infty\}$
by \[
\alpha(x):=\sup\{r>0\,:\, s\mapsto\Phi_{s}(x)\mbox{ is defined on }[0,r]\}.\]
An \emph{admissible time }is a differentiable function $\tau:\mathcal{M}\rightarrow\mathbb{R}$
such that \[
0\leq\tau(x)<\alpha(x)\ \ \ \mbox{for all }x\in\mathcal{M}.\]

Let $\mathcal{F}$ denote a family of subsets of $\mathcal{M}$, and
define \[
\mu:=\inf_{F\in\mathcal{F}}\sup_{x\in F}S(x).\]
Suppose that $\mu\in\mathbb{R}$. We say that $\mathcal{F}$ is \emph{$S$-forward
invariant }if the following holds: if $\tau$ is an admissible time
such that $\tau(x)=0$ if $S(x)\leq\mu-\delta$ for some $\delta>0$
then for all $F\in\mathcal{F}$ the set \[
F_{\tau}:=\left\{ \Phi_{\tau(x)}(x)\,:\, x\in F\right\} \]
 is also a member of $\mathcal{F}$.

For convenience, given a subset $\mathcal{V}\subseteq\mathcal{M}$
and $a\in\mathbb{R}$, let \[
K_{a,\mathcal{V}}:=\mbox{crit}\, S\cap S^{-1}(a)\cap\mathcal{V}\]
denote the set of critical points of $S$ in $\mathcal{V}$ at the
level $a$.\newline

Our main tool will be the following \emph{mountain pass theorem},
whose statement is similar to that of \cite[Proposition 6.3]{Contreras2006}.
In what follows, a \emph{strict local minimizer }of a function $S:\mathcal{M}\rightarrow\mathbb{R}$
is a point $x\in\mathcal{M}$ such that there exists a neighborhood
$\mathcal{N}$ of $x$ such that $S(y)>S(x)$ for all $y\in\mathcal{N}\backslash\{x\}$.
\begin{thm}
\label{thm:mountain pass theorem}Let $\mathcal{M}$ be a Riemannian
Hilbert manifold and $S:\mathcal{M}\rightarrow\mathbb{R}$ a function
of class $C^{2}$. Suppose we are given a sequence $(\mathcal{F}_{n})$
of families of subsets of $\mathcal{M}$ with $\mathcal{F}_{n}\subseteq\mathcal{F}_{n+1}$
for all $n\in\mathbb{N}$. Set $\mathcal{F}_{\infty}:=\bigcup_{n}\mathcal{F}_{n}$.
Set \[
\mu_{\infty}:=\inf_{F\in\mathcal{F}_{\infty}}\sup_{x\in F}S(x).\]
Suppose in addition that:
\begin{enumerate}
\item $\mathcal{F}_{\infty}$ is $S$-forward invariant, and the sets $F\in\mathcal{F}_{\infty}$
are connected;
\item $\mu_{\infty}\in\mathbb{R}$;
\item the flow $\Phi_{s}$ of $-\nabla S$ is relatively complete on $\{\mu_{\infty}-\eta\leq S\leq\mu_{\infty}+\eta\}$
for some $\eta>0$;
\item there are closed subsets $(\mathcal{U}_{h})$ of $\mathcal{M}$ such
that for all $\varepsilon>0$, there exists $n(\varepsilon)\in\mathbb{N}$
such that for all $n\geq n(\varepsilon)$ there exists $F\in\mathcal{F}_{n}$
and $0<\varepsilon_{1}(n)<\varepsilon$ such that \[
F\subseteq\{S\leq\mu_{\infty}-\varepsilon_{1}(n)\}\cup\left(\mathcal{U}_{n}\cap\{S\leq\mu_{\infty}+\varepsilon\}\right);\]

\item there are closed subsets $(\mathcal{V}_{n})$ and a sequence $(r_{n})\subseteq\mathbb{R}^{+}$
such that \[
\mathcal{B}_{r_{n}}(\mathcal{U}_{n}):=\{x\in\mathcal{M}\,:\,\mbox{\emph{dist}}(x,\mathcal{U}_{n})<r_{n}\}\subseteq\mathcal{V}_{n},\]
and such that $S|_{\mathcal{V}_{n}}$ satisfies the Palais-Smale condition
at the level $\mu_{\infty}$.
\end{enumerate}
Then if\[
\mathcal{V}_{\infty}:=\bigcup_{n\in\mathbb{N}}\mathcal{V}_{n},\]
$S$ has a critical point $x\in\mathcal{V}_{\infty}$ with $S(x)=\mu_{\infty}$,
that is, \[
K_{\mu_{\infty},\mathcal{V}_{\infty}}\ne\emptyset.\]
Moreover if \begin{equation}
\sup_{F\in\mathcal{F}_{\infty}}\inf_{x\in F}S(x)<\mu_{\infty}\label{eq:strict local minimizer}\end{equation}
 then there is a point in $K_{\mu_{\infty},\mathcal{V}_{\infty}}$
which is not a strict local minimizer of $S$.
\end{thm}
The proof is an easy application of the following result, which can
be found as \cite[Lemma 6.2]{Contreras2006}.
\begin{lem}
\label{lem:finding an admissable time}Let $\mathcal{M}$ be a Riemannian
Hilbert manifold and $\mathcal{U}\subseteq\mathcal{V}\subseteq\mathcal{M}$
closed subsets such that $\mathcal{B}_{r}(\mathcal{U})\subseteq\mathcal{V}$
for some $r>0$. Let $S:\mathcal{M}\rightarrow\mathbb{R}$ be a $C^{2}$
function, and let $\mu\in\mathbb{R}$ be such that $S|_{\mathcal{V}}$
satisfies the Palais-Smale condition at the level $\mu$. Suppose
in addition that the flow $\Phi_{s}$ of $-\nabla S$ is relatively
complete on $\{\left|S-\mu\right|\leq\eta\}$ for some $\eta>0$. 

Then if $\mathcal{N}$ is any neighborhood of $K_{\mu,\mathcal{V}}$
relative to $\mathcal{V}$, for any $\lambda>0$ there exists $0<\varepsilon<\delta<\lambda$
such that for any $0<\varepsilon_{1}<\varepsilon$ there exists an
admissible time $\tau$ such that \[
\tau(x)=0\mbox{\ \ \ for all }x\in\{\left|S-\mu\right|\geq\delta\},\]
and such that if \[
F:=\{S\leq\mu-\varepsilon_{1}\}\cup\left(\mathcal{U}\cap\{S\leq\mu+\varepsilon\}\right),\]
then \[
F_{\tau}\subseteq\mathcal{N}\cup\{S\leq\mu-\varepsilon_{1}\}.\]
 \end{lem}
\begin{proof}
\emph{(of Theorem \ref{thm:mountain pass theorem})}

We will show that $K_{\mu_{\infty},\mathcal{V}_{n}}\ne\emptyset$
for $n$ large enough. Fix $0<\varepsilon<\delta<\lambda:=1$ as in
the statement of Lemma \ref{lem:finding an admissable time}. By hypothesis
there exists $n(\varepsilon)\in\mathbb{N}$ such that for all $n\geq n(\varepsilon)$
there exists $0<\varepsilon_{1}(n)<\varepsilon$ and $F\in\mathcal{F}_{n}$
such that \[
F\subseteq\{S\leq\mu_{\infty}-\varepsilon_{1}(n)\}\cup\left(\mathcal{U}_{n}\cap\{S\leq\mu_{\infty}+\varepsilon\}\right).\]
For such $n$, $K_{\mu_{\infty},\mathcal{V}_{n}}\ne\emptyset$. Indeed,
if $K_{\mu_{\infty},\mathcal{V}_{n}}=\emptyset$, by Lemma \ref{lem:finding an admissable time},
there exists an admissible time $\tau$ such that $\tau\equiv0$ on
$\{S\leq\mu_{\infty}-\delta\}$, and such that $F_{\tau}$ satisfies
\[
F_{\tau}\subseteq\{S\leq\mu_{\infty}-\varepsilon_{1}(n)\}\]
(for we may take $\mathcal{N}=\emptyset$ in the statement of Lemma
\ref{lem:finding an admissable time}). Since $\mathcal{F}_{\infty}$
is forward invariant, $F_{\tau}\in\mathcal{F}_{\infty}$. This contradicts
the definition of $\mu_{\infty}$. 

To prove the last statement, suppose that $K_{\mu_{\infty},\mathcal{V}_{\infty}}$
consists entirely of strict local minimizers, and \eqref{eq:strict local minimizer}
holds. Choose $\lambda_{0}>0$ such that \[
\sup_{F\in\mathcal{F}_{\infty}}\inf_{x\in F}S(x)<\mu_{\infty}-2\lambda_{0}.\]
For each $x\in K_{\mu_{\infty},\mathcal{V}_{\infty}}$, let $\mathcal{N}(x)$
denote a neighborhood of $x$ such that $S(y)>S(x)$ for all $y\in\mathcal{N}(x)\backslash\{x\}$,
and let \[
\mathcal{N}_{0}:=\bigcup_{x\in K_{\mu_{\infty},\mathcal{V}_{\infty}}}\mathcal{N}(x)\]
and $\mathcal{N}_{n}:=\mathcal{N}_{0}\cap\mathcal{V}_{n}$ for each
$n\in\mathbb{N}$. Let $0<\varepsilon<\delta<\lambda_{0}$ be given
by Lemma \ref{lem:finding an admissable time} for $\mathcal{N}_{0}$.
By hypothesis there exists $n(\varepsilon)\in\mathbb{N}$ such that
for all $n\geq n(\varepsilon)$ there exists $0<\varepsilon_{1}(n)<\varepsilon$
and $F\in\mathcal{F}_{n}$ such that \[
F\subseteq\{S\leq\mu_{\infty}-\varepsilon_{1}(n)\}\cup\left(\mathcal{U}_{n}\cap\{S\leq\mu_{\infty}+\varepsilon\}\right).\]
By Lemma \ref{lem:finding an admissable time}, there exists an admissible
time $\tau$ such that $\tau\equiv0$ on $\{S\leq\mu_{\infty}-\delta\}$
and such that \[
F_{\tau}\subseteq\mathcal{N}_{n}\cup\{S\leq\mu_{\infty}-\varepsilon_{1}(n)\}\subseteq\mathcal{N}_{0}\cup\{S\leq\mu_{\infty}-\varepsilon_{1}(n)\}.\]
By definition of $\mathcal{N}_{0}$, the sets $\mathcal{N}_{0}$ and
$\{S\leq\mu_{\infty}-\varepsilon_{1}(n)\}$ are disjoint, so $\mathcal{N}_{0}\cup\{S\leq\mu_{\infty}-\varepsilon_{1}(n)\}$
is disconnected. Since $F_{\tau}$ is connected by hypothesis, we
either have $F_{\tau}\subseteq\mathcal{N}_{0}$ and $F_{\tau}\cap\{S\leq\mu_{\infty}-\varepsilon_{1}(n)\}=\emptyset$,
or $F_{\tau}\subseteq\{S\leq\mu_{\infty}-\varepsilon_{1}(n)\}$. The
former fails since $\varepsilon_{1}(n)<\varepsilon<\lambda_{0}$,
and the value of $S$ decreases under $\Phi_{s}$, and the latter
contradicts the definition of $\mu_{\infty}$. The proof is complete. 
\end{proof}

\subsection*{Proof of the second statement of Theorem \ref{thm:theorem A}}

$\ $\vspace{6 pt}

The main tool we will use in the proof of the second statement of
Theorem \ref{thm:theorem A} will be Theorem \ref{thm:mountain pass theorem}.
The first step however is the following result, whose statement and
proof closely parallel \cite[Proposition C]{Contreras2006}.
\begin{prop}
\underbar{\label{sub:mountain pass 1-1}}Let $k\in\mathbb{R}^{+}$.
Then there exists a constant $\mu_{0}>0$ such that if $f:I\rightarrow\Lambda_{0}\times\mathbb{R}^{+}$
is any path such that, writing $f(0)=(x_{0},T_{0})$ and $f(1)=(x_{1},T_{1})$,
it holds that:
\begin{enumerate}
\item $S_{k}(x_{0},T_{0})<0$;
\item $x_{1}$ is the constant curve $x_{1}(t)\equiv x_{0}(0)$;
\end{enumerate}
then \[
\sup_{s\in I}S_{k}(f(s))>\mu_{0}>0.\]
\end{prop}
\begin{rem*}
It is important to note that the constant $\mu_{0}$ \textit{\emph{does
}}\textit{not }\textit{\emph{depend on $T_{1}$.}}
\end{rem*}
We shall need the following lemma, taken from \cite[Lemma 5.1]{Contreras2006},
in the proof of Proposition \ref{sub:mountain pass 1-1}. As before,
in the statement of the lemma, $l(x):=\int_{0}^{1}\left|\dot{x}(t)\right|dt$.
\begin{lem}
Let $\theta\in\Omega^{1}(\tilde{M})$. Given any $q\in\tilde{M}$
and any open neighborhood $V\subseteq\tilde{M}$ of $q$, there exists
an open neighborhood $W\subseteq V$ of $q$ and a constant $\beta>0$
such that for any closed curve $x:[0,1]\rightarrow W$ it holds that
\[
\left|\int_{x}\theta\right|\leq\beta l(x)^{2}.\]
\end{lem}
\begin{proof}
\emph{(of Proposition \ref{sub:mountain pass 1-1})}

Fix a point $q\in M$, and choose a neighborhood $W\subseteq M$ of
$q$ small enough such that the conclusion of the lemma above holds.
Pick $\rho\in\mathbb{R}^{+}$ such that \[
0<\rho<\min\left\{ \frac{1}{2}\mbox{diam}\, W,\sqrt{\frac{k}{2\beta^{2}}}\right\} .\]
Write $f(s)=(x_{s},T_{s})$, so $x_{s}\in\Lambda_{0}$ for all $s$.
We claim that there exists $s_{0}\in(0,1)$ such that $l(x_{s_{0}})=\rho$.
Since the functional $s\mapsto l(x_{s})$ is continuous and $l(x_{0})=0$
it suffices to show that there exists $s_{1}\in[0,1)$ such that $l(x_{s_{1}})>\rho$. 

If there exists $s_{1}\in[0,1)$ such that $x_{s_{1}}(I)\varsubsetneq W$
then we are done, since then \[
l(x_{s_{1}})\geq\mbox{dist}(q,W^{c})>\frac{1}{2}\mbox{diam}\, W>\rho.\]
The other possibility is that $x_{s}(I)\subseteq W$ for all $s\in I$.
In this case we claim that we may take $s_{1}=0$, that is, $l(x_{0})>\rho$.
By assumption if $y_{0}(t)=x_{0}(t/T_{0})$ we have\begin{align}
0>S_{k}(x_{0},T_{0}) & =\int_{0}^{1}\frac{1}{2T_{0}}\left|\dot{x}_{0}(t)\right|^{2}dt+kT_{0}-\int_{C(x_{0})}\sigma\nonumber \\
 & \geq\frac{1}{2T_{0}}l(x_{0})^{2}+kT_{0}-\left|\int_{x_{0}}\theta\right|\nonumber \\
 & \geq\left(\frac{1}{2T_{0}}-\beta\right)l(x_{0})^{2}+kT_{0},\label{eq:in W}\end{align}
where the second inequality came from \eqref{eq:cs ineq} and the
third from the previous lemma. From this it follows that $T_{0}>\frac{1}{2\beta}$,
and thus \[
l(x_{0})^{2}>\frac{kT_{0}}{\beta-\frac{1}{2T_{0}}}>\frac{k}{2\beta^{2}}>\rho^{2}.\]
and we are done as before.

We now claim that $S_{k}(f(s_{0}))>0$, which will complete the proof.
Since $x_{s_{0}}\in C_{M}^{\textrm{ac}}(q,q)$ and $l(x_{s_{0}})<\frac{1}{2}\mbox{diam}\, W$,
we have $x_{s_{0}}(I)\subseteq W$. In particular, \eqref{eq:in W}
holds for $x_{s_{0}}$ and so we have \[
S_{k}(f(s_{0}))\geq\left(\frac{1}{2T_{s_{0}}}-\beta\right)\ell^{2}+kT_{s_{0}}=P(T_{s_{0}})\geq\min_{t\in\mathbb{R}^{+}}P(t),\]
where \[
P(t):=\left(\frac{1}{2t}-\beta\right)\rho^{2}+kt.\]
It is elementary to see that \[
\min_{t\in\mathbb{R}^{+}}P(t)=\sqrt{\frac{\rho^{2}}{2k}}=:\mu_{0}>0,\]
and this completes the proof.
\end{proof}
The next lemma will be needed in order to prove relative completeness
of the flow of $-\nabla S_{k}$ on any interval not containing zero.
\begin{lem}
\label{lem:hausdorff lemma}There exists a constant $C>0$ such that
for any $(x_{0},T_{0})\in\Lambda_{M}\times\mathbb{R}^{+}$ and any
$r>0$, if $(x_{1},T_{1})\in\Lambda_{M}\times\mathbb{R}^{+}$ satisfies
\[
\mbox{\emph{dist}}((x_{0},T_{0}),(x_{1},T_{1}))<r,\]
then\[
\left|T_{0}-T_{1}\right|<r\]
and \[
\textrm{\emph{dist}}_{\textrm{\emph{HD}}}(x_{0},x_{1})<Cr.\]

\end{lem}
This result is essentially proved in \cite[Lemma 2.3]{Contreras2006};
there a different metric is used on $\Lambda_{M}\times\mathbb{R}^{+}$
which means an additional condition must be imposed in the statement
of the lemma. In our situation, since we are working with the standard
metric \eqref{eq:standard metric} on $\Lambda_{M}\times\mathbb{R}^{+}$
this additional condition is not needed, and the proof in \cite{Contreras2006}
goes through without any changes.
\begin{cor}
\label{cor:getting the Lm}Let $K\subseteq\tilde{M}$ and $B>0$.
Let \[
\mathcal{U}:=\left\{ (x,T)\in\Lambda_{0}^{K}\times\mathbb{R}^{+}\,:\, T\leq B\right\} .\]
Let $C$ be as in the statement of Lemma \ref{lem:hausdorff lemma}.
Then if $L\subseteq\tilde{M}$ satisfies \[
\left\{ q\in\tilde{M}\,:\,\mbox{\emph{dist}}_{\tilde{g}}(q,q')\leq Cr\mbox{ for some }q'\in K\right\} \subseteq L\]
and we set \[
\mathcal{V}:=\left\{ (x,T)\in\Lambda_{0}^{L}\times\mathbb{R}^{+}\,:\, T\leq B+r\right\} \]
then \[
\mathcal{B}_{r}(\mathcal{U})\subseteq\mathcal{V}.\]
\end{cor}
\begin{proof}
Suppose $(x_{1},T_{1})\in\mathcal{B}_{r}(\mathcal{U})$. Then there
exists $(x_{0},T_{0})\in\mathcal{U}$ with \[
\mbox{dist}((x_{0},T_{0}),(x_{1},T_{1}))<r.\]
By the previous lemma, \[
\mbox{dist}_{\textrm{HD}}(x_{0},x_{1})<Cr\]
and $\left|T_{0}-T_{1}\right|<r$. Thus $(x_{1},T_{1})\in\mathcal{V}$.
\end{proof}
Next, we prove relative completeness of the flow of $-\nabla S_{k}$
on any interval that doesn't contain zero. This proof is very similar
to \cite[Lemma 6.9]{Contreras2006}.
\begin{lem}
\label{lem:relatively complete-1}For all $k\in\mathbb{R}^{+}$, if
$[a,b]\subseteq\mathbb{R}$ is an interval such that $0\notin[a,b]$
then the local flow of $-\nabla S_{k}$ is relatively complete on
$(\Lambda_{0}\times\mathbb{R}^{+})\cap\left\{ a\leq S_{k}\leq b\right\} $.\end{lem}
\begin{proof}
Let $\Phi_{s}:\Lambda_{M}\times\mathbb{R}^{+}\rightarrow\Lambda_{M}\times\mathbb{R}^{+}$
denote the local flow of the vector field $-\nabla S_{k}$. Then for
any $(x,T)\in\Lambda_{M}\times\mathbb{R}^{+}$, \[
S_{k}(\Phi_{s_{1}}(x,T))-S_{k}(\Phi_{s_{2}}(x,T))=\int_{s_{1}}^{s_{2}}\left|\nabla S_{k}(\Phi_{s}(x,T))\right|^{2}ds.\]
By Cauchy-Schwarz inequality we see that \begin{align*}
\mbox{dist}(\Phi_{s_{1}}(x,T),\Phi_{s_{2}}(x,T))^{2} & \leq\left(\int_{s_{1}}^{s_{2}}\left|\nabla S_{k}(\Phi_{s}(x,T))\right|ds\right)^{2}\\
 & \leq\left|s_{1}-s_{2}\right|\int_{s_{1}}^{s_{2}}\left|\nabla S_{k}(\Phi_{s}(x,T))\right|^{2}ds,\end{align*}
and hence\begin{equation}
\mbox{dist}(\Phi_{s_{1}}(x,T),\Phi_{s_{2}}(x,T))^{2}\leq\left|s_{1}-s_{2}\right|\left|S_{k}(\Phi_{s_{1}}(x,T))-S_{k}(\Phi_{s_{2}}(x,T))\right|.\label{eq:cauchy sequence}\end{equation}
Now suppose we are given a pair $(x,T)\in\Lambda_{0}\times\mathbb{R}^{+}$,
such that there exists $a,b\in\mathbb{R}$ with $0\notin[a,b]$ and
\[
a\leq S_{k}(\Phi_{s}(x,T))\leq b\ \ \ \mbox{for all }s\mbox{ such that }\Phi_{s}(x,T)\mbox{ is defined}.\]
Let $[0,\alpha)$ be the maximum interval of definition of $s\mapsto\Phi_{s}(x,T)$,
where $\alpha\in(0,\infty]$. To complete the proof we need to show
$\alpha=\infty$. Suppose to the contrary.

Write $\Phi_{s}(x,T)=(x_{s},T_{s})$. If $s_{n}\uparrow\alpha$ then
$(\Phi_{s_{n}}(x,T))=:(x_{n},T_{n})$ is a Cauchy sequence by \eqref{eq:cauchy sequence}
in $(\Lambda_{0}\times\mathbb{R}^{+})\cap\{a\leq S_{k}\leq b\}$.
Thus $T_{\alpha}:=\lim_{s\uparrow\alpha}T_{s}$ exists and is finite. 

If $T_{\alpha}>0$ then since the sequence $(x_{n},T_{n})$ is Cauchy,
\[
(x_{\alpha},T_{\alpha}):=\lim_{n\rightarrow\infty}(x_{n},T_{n})\]
 exists and is equal to $\Phi_{\alpha}(x,T)$. Since $S_{k}$ is $C^{2}$
we can extend the solution $s\mapsto\Phi_{s}(x,T)$ at $s=\alpha$,
contradicting the definition of $\alpha$. Thus $T_{\alpha}=0$. Hence
there exists a sequence $s_{m}\uparrow\alpha$ such that \[
\frac{d}{ds}T_{s_{m}}\leq0.\]
As before write $x_{m}:=x_{s_{m}}$ and $T_{m}:=T_{s_{m}}$. By \eqref{eq:cauchy sequence}
and Lemma \ref{lem:hausdorff lemma} we may assume there exists a
compact set $K\subseteq\tilde{M}$ such that $(x_{m},T_{m})\subseteq\Lambda_{0}^{K}\times\mathbb{R}^{+}$
for all $m$. If $y_{m}(t):=x_{m}(t/T_{m})$ then \begin{align*}
0 & \geq\frac{d}{ds}T_{m}\\
 & =-\frac{\partial}{\partial T}S_{k}(x_{m},T_{m})\\
 & =\frac{1}{T_{m}}\int_{0}^{T_{m}}\{-k+2E(y_{m},\dot{y}_{m})\}dt\\
 & =-k+\frac{2e_{m}}{T_{m}},\end{align*}
where the penultimate equality used \eqref{eq:d/dT}. Since $\lim_{m\rightarrow\infty}T_{m}=0$,
this forces $\lim_{m\rightarrow\infty}e_{m}=0$. As in the proof of
the second part of Theorem \ref{thm:the Palais Smale theorem}, this
implies $S_{k}(x_{m},T_{m})\rightarrow0$, contradicting the fact
that $0\notin[a,b]$. This implies that we must have originally had
$\alpha=\infty$, and so completes the proof.
\end{proof}
We now move towards proving the second statement of Theorem \ref{thm:theorem A}.
In fact, we will prove the following stronger result, which is based
on \cite[Proposition 7.1]{Contreras2006}.
\begin{prop}
\label{thm:stronger}Let $c=c(g,\sigma)\in\mathbb{R}\cup\{\infty\}$.
For almost all $k\in(0,c)$ there exists a contractible closed orbit
of $\phi_{t}$ with energy $k$. Moreover this orbit has positive
$S_{k}$-action, and is not a strict local minimizer of $S_{k}$ on
$\Lambda_{0}\times\mathbb{R}^{+}$. This holds for a specific $k\in(0,c)$
if $S_{k}$ is known to satisfy the Palais-Smale condition on the
level $k$. \end{prop}
\begin{proof}
Fix $k_{0}\in\mathbb{R}^{+}$. There exists $(x_{0},T_{0})\in\Lambda_{0}\times\mathbb{R}^{+}$
such that $S_{k_{0}}(x_{0},T_{0})<0$. Indeed, there exists a closed
curve $\tilde{y}:[0,T_{0}]\rightarrow\tilde{M}$ such that $A_{k_{0}}(\tilde{y})<0$.
Then the projection $y:[0,T_{0}]\rightarrow M$ of $\tilde{y}$ to
$M$ is a closed curve, and if $x_{0}(t):=y(tT_{0})$ then $(x_{0},T_{0})\in\Lambda_{0}\times\mathbb{R}^{+}$
and $S_{k_{0}}(x_{0},T_{0})=A_{k_{0}}(\tilde{y})<0$. There exists
$\varepsilon>0$ such that for all $k\in J:=[k_{0},k_{0}+\varepsilon]$
we have $S_{k}(x_{0},T_{0})<0$. 

Let $x_{1}$ denote the constant loop at $x_{0}(0)$. Given $k\in J$,
let $\mu_{0}(k)>0$ be the constant given by Proposition \ref{sub:mountain pass 1-1}
such that any path $f\in C^{0}(I,\Lambda_{0}\times\mathbb{R}^{+})$
with $f(0)=(x_{0},T_{0})$ and $f(1)=(x_{1},T)$ for some $T>0$ satisfies
\[
\sup_{s\in I}S_{k}(f(s))>\mu_{0}(k).\]
Choose $T_{1}>0$ such that \[
T_{1}<\inf_{k\in J}\frac{\mu_{0}(k)}{k}.\]
Then \[
\max\{S_{k}(x_{0},T_{0}),S_{k}(x_{1},T_{1})\}=kT_{1}<\mu_{0}(k)\ \ \ \mbox{for all }k\in J.\]

Set\[
\Gamma:=\left\{ f\in C^{0}(I,\Lambda_{0}\times\mathbb{R}^{+})\,:\, f(0)=(x_{0},T_{0}),f(1)=(x_{1},T_{1})\right\} .\]
Let $(K_{n})\subseteq\tilde{M}$ denote compact sets such that $K_{n}\subseteq K_{n+1}$
and $\bigcup_{n}K_{n}=\tilde{M}$. Let \[
\Gamma_{n}:=\Gamma\cap C^{0}(I,\Lambda_{0}^{K_{n}}\times\mathbb{R}^{+}).\]

Define for $k\in J$,\[
\mu_{n}(k):=\inf_{f\in\Gamma_{n}}\sup_{s\in I}S_{k}(f(s));\]
\[
\mu_{\infty}(k):=\inf_{f\in\Gamma}\sup_{s\in I}S_{k}(f(s)).\]
Then $\mu_{n}(k)\geq\mu_{n+1}(k)\geq\mu_{\infty}(k)\geq\mu_{0}(k)$
for all $n\in\mathbb{N}$ and $k\in J$, and the functions $\mu_{n}:J\rightarrow\mathbb{R}$
converge pointwise to $\mu_{\infty}$. Moreover both $\mu_{n}$ and
$\mu_{\infty}$ are non-decreasing. Since $\mu_{\infty}$ is non-decreasing,
by Lebesgue's theorem there exists a subset $J_{0}\subseteq(k_{0},k_{0}+\varepsilon)$
with $J\backslash J_{0}$ having measure zero such that $\mu_{\infty}|_{J_{0}}$
is locally Lipschitz. In other words, for all $j\in J_{0}$ there
exist constants $M(j)>0$ and $\delta(j)>0$ such that for all $\left|\delta\right|<\delta(j)$
it holds that \[
\left|\mu_{\infty}(j+\delta)-\mu_{\infty}(j)\right|<M(j)\left|\delta\right|.\]
Fix $j\in J_{0}$ and a sequence $(j_{m})\subseteq J_{0}$ with $j_{m}\downarrow j$.
Let $f_{n,m}\in\Gamma_{n}$ be paths such that \[
\max_{s\in I}S_{j_{m}}(f_{n,m}(s))\leq\mu_{n}(j_{m})+(j_{m}-j).\]
Next, define\[
\mathcal{U}_{n}:=\left\{ (x,T)\in\Lambda_{0}^{K_{n}}\times\mathbb{R}^{+}\,:\, T\leq M(j)+2\right\} .\]
Choose another collection $(L_{n})\subseteq\tilde{M}$ of compact
sets such that $K_{n}\subseteq L_{n}$, and such that if\[
\mathcal{V}_{n}:=\left\{ (x,T)\in\Lambda_{0}^{L_{n}}\times\mathbb{R}^{+}\,:\, T\leq M(j)+3\right\} ,\]
then $\mathcal{B}_{1}(\mathcal{U}_{n})\subseteq\mathcal{V}_{n}$.
Such a collection $(L_{n})$ exists by Corollary \ref{cor:getting the Lm}.
Since $\mu_{\infty}(j)\ne0$, from Proposition \ref{thm:(the-key-theorem new)}
it follows that $S_{j}|_{\mathcal{V}_{n}}$ satisfies the Palais-Smale
condition at the level $\mu_{\infty}(j)$ for all $n\in\mathbb{N}$. 

Since $k\mapsto S_{k}(x,T)$ is increasing,\begin{equation}
\max_{s\in I}S_{j}(f_{n,m}(s))\leq\max_{s\in I}S_{j_{m}}(f_{n,m}(s))\leq\mu_{n}(j_{m})+(j_{m}-j).\label{eq:increasing}\end{equation}
If $s\in I$ is such that \[
S_{j}(f_{n,m}(s))>\mu_{\infty}(j)-(j_{m}-j),\]
writing $f_{n,m}(s)=:(x_{s}^{n,m},T_{s}^{n,m})$ we have: \begin{align*}
T_{s}^{n,m} & =\frac{S_{j_{m}}(f_{n,m}(s))-S_{j}(f_{n,m}(s))}{j_{m}-j}\\
 & \leq\frac{\mu_{\infty}(j_{m})-\mu_{n}(j)}{j_{m}-j}+2\\
 & \leq\frac{\mu_{\infty}(j_{m})-\mu_{\infty}(j)}{j_{m}-j}+2\\
 & \leq M(j)+2,\end{align*}
for $n$ large enough. 

Given $\varepsilon>0$, first choose $m$ large enough such that \[
j_{m}-j<\frac{\varepsilon}{2(M(j)+1)},\]
 and then select $n$ large enough so that \[
\mu_{n}(j_{m})-\mu_{\infty}(j_{m})<\varepsilon/2.\]
Then \begin{align*}
\mu_{n}(j_{m})-\mu_{\infty}(j)+(j_{m}-j) & =\left\{ \mu_{n}(j_{m})-\mu_{\infty}(j_{m})\right\} +\left\{ \mu_{\infty}(j_{m})-\mu_{\infty}(j)\right\} +(j_{m}-j)\\
 & <\varepsilon/2+M(j)(j_{m}-j)+(j_{m}-j)\\
 & <\varepsilon.\end{align*}
Then by \eqref{eq:increasing}, \[
f_{n,m}(I)\subseteq\left\{ S_{j}\leq\mu_{\infty}(j)-(j_{m}-j)\right\} \cap\left(\mathcal{U}_{n}\cap\left\{ S_{j}\leq\mu_{\infty}(j)+\varepsilon\right\} \right).\]
Since $\mu_{\infty}(j)\ne0$, by Lemma \ref{lem:relatively complete-1}
the gradient flow of $-S_{j}$ is relatively complete on $\{\mu_{\infty}(j)-\eta\leq S_{j}\leq\mu_{\infty}(j)+\eta\}$
for some $\eta>0$. Theorem \ref{thm:mountain pass theorem} then
gives a critical point for $S_{j}|_{\Lambda_{0}\times\mathbb{R}^{+}}$
which is not a strict local minimizer (we are applying Theorem \ref{thm:mountain pass theorem}
with $\mathcal{F}_{n}:=\{f(I)\,:\, f\in\Gamma_{n}\}$).\newline

Finally, suppose that $k<c(g,\sigma)\leq\infty$ is such that $S_{k}$
satisfies the Palais-Smale condition. In this case the theorem is
immediate from Lemma \ref{lem:relatively complete-1} and Theorem
\ref{thm:mountain pass theorem}. Indeed, by Lemma \ref{lem:relatively complete-1}
we may simply take $\mathcal{U}_{n}=\mathcal{V}_{n}=\mathcal{V}_{\infty}=\Lambda_{0}\times\mathbb{R}^{+}$,
as then the hypotheses of Theorem \ref{thm:mountain pass theorem}
are trivially satisfied. This completes the proof of Theorem \ref{thm:theorem A}.
\end{proof}
\bibliographystyle{amsplain}
\bibliography{C:/Users/will/willbibtex}

\end{document}